\documentclass[a4paper]{article}

\usepackage[utf8]{inputenc}
\usepackage[width=15cm,height=22cm]{geometry}
\usepackage{amsmath,amsfonts,amssymb,amsthm}
\usepackage[hidelinks,pdfusetitle]{hyperref}
\usepackage[capitalise]{cleveref}
\usepackage[usenames,dvipsnames]{xcolor}
\usepackage{bbm}
\usepackage{stmaryrd}

\crefname{equation}{}{}

\newtheorem{proposition}{Proposition}[section]
\newtheorem{theorem}[proposition]{Theorem}

\newtheorem{definition}[proposition]{Definition}
\newtheorem{corollary}[proposition]{Corollary}
\newtheorem{lemma}[proposition]{Lemma}
\newtheorem{remark}[proposition]{Remark}
\numberwithin{equation}{section}
\newtheorem{example}[proposition]{Example}

\newcommand*\samethanks[1][\value{footnote}]{\footnotemark[#1]}

\title{Convergent realizations of Lie subalgebras}

\author{Karine Beauchard\texorpdfstring{\thanks{Univ Rennes, CNRS, IRMAR -- UMR 6625, F-35000 Rennes, France}}{}, Jérémy Le Borgne\texorpdfstring{\samethanks}{},
Fr\'ed\'eric Marbach\texorpdfstring{\thanks{DMA, École normale supérieure, Université PSL, CNRS, 75005 Paris, France}}{}}

\newcommand{\N}{\mathbb{N}}
\newcommand{\K}{\mathbb{K}}
\newcommand{\C}{\mathbb{C}}
\newcommand{\R}{\mathbb{R}}

\newcommand{\dd}{\,\mathrm{d}}

\newcommand{\B}{\mathcal{B}}

\newcommand{\fg}{\mathfrak{g}}
\newcommand{\fh}{\mathfrak{h}}
\newcommand{\fp}{\mathfrak{p}}
\newcommand{\fhp}{\mathfrak{h}_{\fp}}
\newcommand{\Id}{\operatorname{Id}}
\newcommand{\im}{\operatorname{im}}

\newcommand{\Der}{\operatorname{Der}}
\newcommand{\End}{\operatorname{End}}
\newcommand{\Fac}{\operatorname{Fac}}
\newcommand{\codim}{\operatorname{codim}}
\newcommand{\ord}{\operatorname{ord}}
\newcommand{\rank}{\operatorname{rank}}
\newcommand{\GL}{\operatorname{GL}}
\newcommand{\kx}{\K \llbracket x \rrbracket}
\newcommand{\ky}{\K \llbracket y \rrbracket}
\newcommand{\lxl}{\llbracket x \rrbracket}

\DeclareMathOperator{\Br}{Br}
\DeclareMathOperator{\ad}{ad}

\DeclareMathOperator{\vect}{span}

\newcommand{\opnorm}[1]{\left|\mkern-1.5mu\left|\mkern-1.5mu\left| #1 \right|\mkern-1.5mu\right|\mkern-1.5mu\right|}

\newcommand{\intset}[1]{\llbracket #1 \rrbracket}

\begin{document}

\maketitle

\begin{abstract}
    It has been known since the seminal work of Guillemin and Sternberg that Lie subalgebras of finite codimension of an arbitrary real or complex Lie algebra can be realized as subalgebras of formal vector fields over formal power series.
    In this note, we characterize the Lie subalgebras which admit a convergent realization in the sense of locally analytic vector fields.
    We give generalizations of these properties for the problem of output realization.
    
    We give reformulations and applications of these algebraic results in the context of control theory.
    In particular, we recover and clarify previous results on the realization of Chen--Fliess series for control-affine systems, the equivalence of control systems, and the existence of embedded or canonical systems.
\end{abstract}

\setcounter{tocdepth}{1}
\tableofcontents

\newpage

\section{Introduction}
\label{sec:realizations-Lie}

\subsection{Power series and vector fields}
\label{sec:intro-powers}

We define our main objects of study, and refer to \cite{Draisma2012,Sambale2023} for a more thorough introduction.
Throughout this text, $\K$ denotes $\R$ or $\C$ and is the ground field of all vector spaces.
For a positive integer $d \in \N^*$, we consider the indeterminates $x := (x_1, \dotsc, x_d)$.
We denote by $\kx$ the $\K$-algebra of \emph{formal commutative power series} in the $x_i$ with coefficients in $\K$.
Its elements are series of the form
\begin{equation}
    f = \sum_{m \in \N^d} f_m x^m
\end{equation}
where $x^m := x_1^{m_1} \dotsb x_d^{m_d}$ and $f_m \in \K$.
For a multi-index $m \in \N^d$, we write $|m| := m_1 + \dotsb + m_d$ and $m! := m_1! \dotsb m_d!$.
We say that such a power series $f$ is \emph{convergent} when 
\begin{equation} \label{eq-def:ps-conv}
    \exists C_f \geq 0, \forall m \in \N^d, \quad |f_m| \leq C_f^{|m|+1}.
\end{equation}
This corresponds to the fact that $f$ can be seen as an analytic map in the ball of radius $1/C_f$.

We also consider $\Der \kx$, the Lie algebra of derivations on $\kx$, whose elements are \emph{formal vector fields in $d$ variables} (see \cite[p.\ 40]{Draisma2012}), which can thus be written as
\begin{equation}
    v = \sum_{i=1}^d v^i \partial_i
\end{equation}
where $v^i \in \kx$ and $\partial_i$ is the differentiation with respect to $x_i$.
We say that such a vector field is \emph{convergent} when each of its components $v^i$ satisfies \eqref{eq-def:ps-conv}.
For a vector field $v \in \Der\kx$, we use the notation $v(0) := (v^1_0, \dotsc, v^d_0) \in \K^d$.
We denote by $\Der_0 \kx$ the Lie subalgebra of formal vector fields vanishing at $0$ (i.e.\ such that $v(0) = 0$).

\subsection{Formal realizations of Lie subalgebras}

Let $\fg$ be a Lie algebra over $\K$, and $\fh$ a vector subspace of $\fg$.

\begin{definition}
    The pair $(\fg, \fh)$ is called \emph{transitive} when $\fh$ is a Lie subalgebra and $\codim_{\fg} \fh < +\infty$.
\end{definition}

\begin{definition}
    We say that $\rho : \fg \to \Der \kx$ is a \emph{realization in $d$ variables} of the pair $(\fg, \fh)$ when $\rho$ is a homomorphism of Lie algebras such that $\fh = \rho^{-1}(\Der_0 \kx)$.
\end{definition}

Let $\rho$ be a realization in $d$ variables of a pair $(\fg, \fh)$. 
Since $\rho$ is a homomorphism of Lie algebras and $\Der_0 \kx$ is a Lie subalgebra of $\Der \kx$, $\fh$ is a Lie subalgebra of $\fg$.
Since $\fh$ is the kernel of the linear map $g \mapsto \rho(g)(0)$ from $\fg$ to $\K^d$, $\codim_\fg \fh \leq d < +\infty$.
Hence, $(\fg,\fh)$ is a transitive pair.
Conversely, it has been known since Guillemin and Sternberg \cite{GuilleminSternberg1964} that any transitive pair admits a realization in minimal dimension, which is unique up to a change of coordinates (see \cref{thm:uniqueness}).

\begin{theorem}
    \label{thm:transitive}
    Any transitive pair $(\fg,\fh)$ admits a realization in $d := \codim_\fg \fh$ variables.
\end{theorem}

\subsection{An explicit realization formula}

Initially, Guillemin and Sternberg's realization theorem of \cite{GuilleminSternberg1964} is an abstract existence result.
Nevertheless, Draisma extracted in \cite{Draisma2002} from Blattner's subsequent proof of \cite{Blattner1970} an explicit realization formula, which we now state.
We start by recalling the following important result.

Given a Lie algebra $\fg$, we denote by $U(\fg)$ the universal enveloping algebra of $\fg$.

\begin{proposition}[Poincaré--Birkhoff--Witt] \label{thm:PBW}
    Let $\mathcal{B}$ be a totally ordered basis of $\fg$. 
    Then the products $b_p^{n_p} \dotsb b_1^{n_1}$, for $p \in \N$, $b_p > \dotsb > b_1 \in \mathcal{B}$, $n_1, \dotsc, n_p \in \N^*$ (including the empty product $1$) form a basis of $U(\fg)$, called the \emph{Poincaré--Birkhoff--Witt basis of $U(\fg)$ associated with $\mathcal{B}$}.
\end{proposition}

\begin{theorem} \label{thm:blattner}
    Let $(\fg, \fh)$ be a transitive pair and $d := \codim_{\fg} \fh$.
    Let $\mathcal{B}$ be a totally ordered basis of $\fg$ such that $\mathcal{B} = \{ b_1, \dotsc, b_d \} \cup \mathcal{B}'$ where $\mathcal{B}'$ is a basis of $\fh$, with $b_1 < \dotsb < b_d$ and $b_d < b$ for every $b \in \mathcal{B}'$.
    Define $\rho : \fg \to \Der \kx$ by
    \begin{equation}
        \label{eq:formula-blattner}
        \rho(g) := \sum_{i=1}^d \sum_{m \in \N^d} \frac{\chi_i(\mathbf{b}^m g)}{m!} x^m \partial_i
    \end{equation}
    where $\mathbf{b}^m := b_d^{m_d} \dotsb b_1^{m_1}$ and $\chi_i(u)$ denotes the component along $b_i$ of the decomposition of $u \in U(\fg)$ in the Poincaré--Birkhoff--Witt basis associated with $\mathcal{B}$.
    Then $\rho$ is a realization of $(\fg,\fh)$.
\end{theorem}

We give a proof of \cref{thm:blattner} in \cref{sec:proof-formula} (which obviously implies \cref{thm:transitive}).

\begin{remark}
    As noted in \cite[Section 2]{Draisma2002}, the realization defined in \cref{eq:formula-blattner} of course depends on the choice of $\{ b_1, \dotsc, b_d \}$ spanning a complement of $\fh$ in $\fg$, but not on $\mathcal{B}'$, the choice of basis for~$\fh$.
    Indeed, given $\{ b_1, \dotsc, b_d \}$ and using the Poincaré--Birkhoff--Witt theorem, one can write
    \begin{equation}
        \label{eq:Ug=I+}
        U(\fg) = \left( \bigoplus_{n \in \N^d} \K \mathbf{b}^n \right) \oplus I
        \quad \text{where} \quad
        I := \fh U(\fg) := \vect_{\K} \{ h a \mid h \in \fh, a \in U(\fg)\}.
    \end{equation}
    Thus, looking at \eqref{eq:formula-blattner} and taking $g \in \fg$ and $m \in \N^d$, computing $\chi_i(\mathbf{b}^m g)$ can be done through decompositions which only depend on $b_1,\dotsc,b_d$ and $\fh$ (but not on the chosen basis $\mathcal{B}'$ of $\fh$).
\end{remark}

\subsection{A criterion for convergent realizations}

Motivated by applications to control theory (see \cref{sec:applications}), we investigate the following notion.

\begin{definition}
    Let $\rho : \fg \to \Der \kx$ be a realization of a transitive pair $(\fg,\fh)$.
    We say that it is \emph{convergent} when, for all $g \in \fg$, $\rho(g)$ is a convergent formal vector field.
\end{definition}

Our main contribution is the following characterization of the transitive pairs which admit a convergent realization, and the convergence of the explicit Blattner--Draisma realizations.
For a transitive pair $(\fg,\fh)$ let $\chi : \fg \to \fg / \fh$ be the quotient map. 
We fix any norm on the finite-dimensional vector space $\fg / \fh$. 
The following condition is independent of this choice of norm:
\begin{equation}
    \label{eq:growth}
    \begin{split}
        \forall F\subset\fg \text{ finite}, \enskip
        \exists C_F>0, \enskip
        & \forall n\in\N^*, \enskip
        \forall g_1,\dotsc,g_n\in F, \\
        &
        \| \chi ([g_1, [\dotsc, [g_{n-1}, g_n]\dotsb]]) \|_{\fg / \fh}
        \leq
        C_F^n (n-1)!.
    \end{split}
\end{equation}

\begin{theorem} \label{thm:main}
    A transitive pair $(\fg,\fh)$ has a convergent realization if and only if \eqref{eq:growth} holds.
    When \eqref{eq:growth} holds, the realizations given by \cref{thm:blattner} are convergent.
\end{theorem}

We prove the sufficiency of condition \eqref{eq:growth} to obtain a convergent realization in \cref{sec:sufficient}, and its necessity in \cref{sec:necessary}.
We give applications of this result to control theory in \cref{sec:applications}.

The convergence of the realization was already known to hold in the case of finite-dimensional~$\fg$, without any additional assumption.
Our general characterization allows to recover the following result, due to Draisma in \cite[Proposition 2.5]{Draisma2002}.

\begin{corollary} \label{cor:draisma}
    Let $(\fg, \fh)$ be a transitive pair.
    Assume that $\fg$ is finite-dimensional.
    Then $(\fg, \fh)$ has a convergent realization.
\end{corollary}

\begin{proof}
    Let $F$ be a finite subset of $\fg$.
    Let $\| \cdot \|$ be a norm on $\fg$.
    Since $\fg$ is finite-dimensional, the linear maps $\ad_{g} : a \mapsto [g,a]$ are continuous on $\fg$ and there exists $M > 0$ such that, for every $g \in F$ and $a \in \fg$, $\| [g, a] \| \leq M \| a \|$.
    Hence, for every $n \in \N^*$ and $g_1, \dotsc, g_n \in F$,
    \begin{equation}
        \| \chi([g_1, [g_2, [\dotsc, [g_{n-1}, g_n]\dotsb]]]) \|_{\fg / \fh}
        \leq 
        \| [g_1, [g_2, [\dotsc, [g_{n-1}, g_n]\dotsb]]] \|
        \leq M^{n-1} \| g_n \|,
    \end{equation}
    where we used the quotient norm on $\fg/\fh$ so that $\| \chi(g) \|_{\fg/\fh} \leq \| g \|$ for every $g \in \fg$.
    Thus \eqref{eq:growth} holds (even without the factorial term) with $C_F := \max \{ M, \max_{g \in F} \| g \| \}$ so that \cref{cor:draisma} indeed follows from \cref{thm:main}.
\end{proof}

\section{Proofs of the existence and convergence of realizations}

\subsection{Proof of the realization formula}
\label{sec:proof-formula}

We give a direct computational proof of \cref{thm:blattner}.
It is inspired by \cite{Draisma2002,Draisma2012}, but formulated by encoding the realization $\rho$ into a generating series in $U(\fg)\lxl$.
We do not assume that $\dim \fg < +\infty$.

\begin{proof}[Proof of \cref{thm:blattner}]
    Let $I := \fh U(\fg) \subset U(\fg)$ as in \eqref{eq:Ug=I+}.
    It is a right ideal of $U(\fg)$.
    
    For $1 \leq i \leq d$, the linear form $\chi_i : U(\fg) \to \K$ extracts the coefficient of $b_i=\mathbf{b}^{e_i}$ in the decomposition \eqref{eq:Ug=I+}. 
    In particular, $\chi_i = 0$ on $I$.
    
    We extend $\chi_i$ coefficientwise to a $\K$-linear map $\chi_i : U(\fg)\lxl \to \kx$.
    
    We now introduce the generating series
    \begin{equation}
        \label{eq:gen-E}
        \mathcal{E}(x)
        :=
        \sum_{m \in \N^d} \mathbf{b}^m \frac{x^m}{m!}
        =
        e^{x_d b_d}\dotsm e^{x_1 b_1}
        \in U(\fg)\lxl.
    \end{equation}
    Note that $\mathcal{E}(x)$ and each of the $e^{x_i b_i}$ are invertible in $U(\fg)\lxl$.
    
    For $g \in \fg$, the formal vector field defined in \eqref{eq:formula-blattner} is precisely
    \begin{equation}
        \label{eq:rho=E}
        \rho(g)=\sum_{i=1}^d \chi_i(\mathcal{E}(x)g) \partial_i.
    \end{equation}

    \medskip \noindent \textbf{Step 1: decomposition modulo $\fh$.}
    We prove that, for every $g \in \fg$,
    \begin{equation} \label{eq:key-congruence}
        \mathcal{E}(x)g
        \equiv
        \sum_{j=1}^d \chi_j(\mathcal{E}(x) g) \partial_j \mathcal{E}(x)
        \pmod{I\lxl},
    \end{equation}
    where $I\lxl$ denotes the subspace of $U(\fg)\lxl$ consisting of formal series whose coefficients lie in $I$.
    
    For $1 \leq j \leq d$, define $Y_j(x) \in U(\fg)\lxl$ by $\partial_j \mathcal{E}(x) = Y_j(x)\mathcal{E}(x)$.
    A direct differentiation of the product $\mathcal{E}(x)=e^{x_d b_d}\dotsm e^{x_1 b_1}$
    gives
    \begin{equation} \label{eq:explicit-Yj}
        Y_j(x)
        =
        e^{x_d b_d}\dotsm e^{x_{j+1} b_{j+1}}
        b_j
        e^{-x_{j+1} b_{j+1}}\dotsm e^{-x_d b_d}.
    \end{equation}
    For $a,b\in\fg$, in $U(\fg)\lxl$, one has
    \begin{equation}
        \label{eq:conjugation}
        e^{xb} a e^{-xb}
        = e^{x\ad_b}(a)
        = \sum_{n\geq 0} \frac{x^n}{n!}\ad_b^n(a) \in \fg\lxl.
    \end{equation}
    By repeated use of \eqref{eq:conjugation}, $Y_j(x)\in\fg\lxl$.
    Moreover, $Y_j(0) = b_j$.
    
    Similarly, for $g \in \fg$, define
    \begin{equation}
        \widetilde g(x) := \mathcal{E}(x) g \mathcal{E}(x)^{-1}.
    \end{equation}
    Again, using \eqref{eq:conjugation}, $\widetilde g(x)\in \fg\lxl$.

    We extend the quotient map $\chi : \fg \to (\fg/\fh)$ coefficientwise to a $\K$-linear map $\fg\lxl \to (\fg/\fh)\lxl$.
    Since $(\chi(b_1),\dotsc,\chi(b_d))$ is a basis of $\fg/\fh$, there exist unique coefficients $A_{ij}(x)\in\kx$ such that
    \begin{equation}
        \chi(Y_j(x))
        =
        \sum_{i=1}^d A_{ij}(x)\chi(b_i).
    \end{equation}
    Since $Y_j(0) = b_j$, we have $A_{ij}(0)=\delta_{ij}$, hence $\det A(0)=1$. 
    Thus $\det A(x)$ has nonzero constant term, so it is invertible in $\kx$. 
    By the adjugate formula, $A(x)$ is invertible in $M_d(\kx)$.
    Hence $\chi(Y_1(x)),\dotsc,\chi(Y_d(x))$ form a basis of $(\fg/\fh)\lxl$ over $\kx$. 
    
    It follows that, for every $g \in \fg$, there exist unique series $w_1(x),\dotsc,w_d(x)\in\kx$ and an element $H_g(x)\in\fh\lxl$ such that
    \begin{equation} \label{eq:decomposition-gtilde}
        \widetilde g(x)
        =
        \sum_{j=1}^d w_j(x) Y_j(x) + H_g(x).
    \end{equation}
    Multiplying on the right by $\mathcal{E}(x)$, and using $\partial_j \mathcal{E}(x) = Y_j(x) \mathcal{E}(x)$, we obtain
    \begin{equation} \label{eq:decomposition-Eg}
        \mathcal{E}(x)g
        =
        \sum_{j=1}^d w_j(x)\partial_j \mathcal{E}(x) + H_g(x)\mathcal{E}(x).
    \end{equation}
    Since $H_g(x)\in\fh\lxl$ and $I=\fh U(\fg)$ is a right ideal, we have $H_g(x)\mathcal{E}(x)\in I\lxl$.
    Therefore
    \begin{equation} \label{eq:key-congruence-w}
        \mathcal{E}(x)g
        \equiv
        \sum_{j=1}^d w_j(x)\partial_j \mathcal{E}(x)
        \pmod{I\lxl}.
    \end{equation}
    Since $\chi_i$ acts coefficientwise on $U(\fg)\lxl$, it commutes with the formal derivations $\partial_j$.
    Moreover, $\chi_i(\mathcal{E}(x)) = x_i$ so $\chi_i(\partial_j \mathcal{E}(x)) = \partial_j(\chi_i(\mathcal{E}(x))) = \delta_{ij}$.
    Since $\chi_i$ vanishes on $I$, applying $\chi_i$ to \eqref{eq:key-congruence-w} yields $\chi_i (\mathcal{E}(x) g) = w_i(x)$.
    Hence \eqref{eq:key-congruence} holds.
    
    \medskip \noindent \textbf{Step 2: the Lie algebra property.}
    Let $g_1,g_2 \in \fg$.
    Since $I\lxl$ is stable under right multiplication by constant elements of $U(\fg)$, multiplying \eqref{eq:key-congruence} for $g_1$ on the right by $g_2$ gives
    \begin{equation}
        \mathcal{E}(x)g_1g_2
        \equiv
        \sum_{j=1}^d \chi_j (\mathcal{E}(x) g_1) (\partial_j \mathcal{E}(x))g_2
        \pmod{I\lxl}.
    \end{equation}
    Because $g_2$ is constant with respect to $x$, we have $(\partial_j \mathcal{E}(x))g_2=\partial_j(\mathcal{E}(x)g_2)$.
    Applying $\chi_i$,
    \begin{equation} \label{eq:chi-g1g2}
        \chi_i(\mathcal{E}(x)g_1g_2)
        =
        \sum_{j=1}^d \chi_j (\mathcal{E}(x) g_1) \partial_j (\chi_i(\mathcal{E}(x)g_2)).
    \end{equation}
    Exchanging $g_1$ and $g_2$, and subtracting, we obtain
    \begin{equation} \label{eq:chi-g2g1}
        \chi_i(\mathcal{E}(x)[g_1, g_2])
        =
        \sum_{j=1}^d \chi_j (\mathcal{E}(x) g_1) \partial_j (\chi_i(\mathcal{E}(x)g_2)) - \chi_j (\mathcal{E}(x) g_2) \partial_j (\chi_i(\mathcal{E}(x)g_1)).
    \end{equation}
    Recalling \eqref{eq:rho=E}, $\rho([g_1,g_2])=[\rho(g_1),\rho(g_2)]$.
    Thus $\rho : \fg \to \Der \kx$ is a Lie algebra homomorphism.
    
    \medskip \noindent \textbf{Step 3: identification of the isotropy.}
    Evaluating at $x=0$, we have $\mathcal{E}(0)=1$.
    By \eqref{eq:rho=E}, $\rho(g)(0) = (\chi_1(g), \dotsc, \chi_d(g))$.
    Since $(\chi(b_1),\dotsc,\chi(b_d))$ is a basis of $\fg/\fh$, we have
    \begin{equation}
        g \in \fh
        \iff
        \chi_i(g)=0 \text{ for all } i \in \intset{1,d}
        \iff
        \rho(g)(0)=0
        \iff
        \rho(g)\in \Der_0 \kx.
    \end{equation}
    Hence $\fh=\rho^{-1}(\Der_0 \kx)$.
    This concludes the proof.
\end{proof}

\subsection{Proof of the sufficiency of the convergence criterion}
\label{sec:sufficient}

In this subsection, we prove that, if a transitive pair satisfies estimate \eqref{eq:growth}, then the formal realization given by \cref{thm:blattner} is convergent, which entails the reverse implication of \cref{thm:main}.

We start with the following lemma, which allows to ``bubble up'' an element $a$ from right to left, collecting commutators along the way, as in the identity
\begin{equation}
    a_2 a_1 a = a a_2 a_1 + [a_1,a] a_2 + [a_2,a] a_1 + [a_2,[a_1,a]].
\end{equation}

\begin{lemma} \label{lem:bubble}
    Let $\mathfrak{a}$ be an associative algebra.
    Let $n \in \N^*$ and $a, a_1, \dotsc, a_n \in \mathfrak{a}$.
    For $J = \{ j_1 < \dots < j_h \} \subset \intset{1,n}$, we use the notations: $\intset{1,n} \setminus J = \{t_1 < \dots < t_{\ell} \}$ and
    \begin{equation}
        \mu_{J}(a):=[a_{t_{\ell}}, \dots [a_{t_1}, a ]\dotsb ].
    \end{equation}
    Then the following identity holds:
    \begin{equation}
        a_n \dotsb a_1 a = \sum_{J \subset \intset{1,n}} \mu_{J}(a) a_{j_h}\cdots a_{j_1}.
    \end{equation}
\end{lemma}

\begin{proof}
    We proceed by induction on $n \in \N^*$.
    For $n = 1$, one has $a_1 a = [a_1, a] + a a_1 = \mu_{\emptyset}(a) + \mu_{\{1\}}(a) a_1$. 
    Let $n \in \N^*$. 
    We assume the identity proved up to $n$ and we prove it for $(n+1)$.
    To be clear, we write $\mu_J^n$ instead of $\mu_J$.
    By the induction hypothesis and the $n = 1$ case,
    \begin{equation}
        \begin{split}
            a_{n+1} a_n \dotsb a_1 a 
            & = a_{n+1} \sum_{J \subset \intset{1,n}} \mu^n_J(a) a_{j_h} \dotsb a_{j_1} \\
            & = \sum_{J \subset \intset{1,n}} ([a_{n+1}, \mu^n_J(a)] + \mu^n_J(a) a_{n+1}) a_{j_h} \dotsb a_{j_1} \\
            & = \sum_{J' \subset \intset{1,n+1}} \mu^{n+1}_{J'}(a) a_{j'_h} \dotsb a_{j'_1},
        \end{split}
    \end{equation}
    which gives the conclusion by partitioning the subsets $J'$ of $\intset{1,n+1}$ depending on whether $n+1\in J'$.
\end{proof}

From now on, we assume that $(\fg,\fh)$ is a transitive pair satisfying estimate \eqref{eq:growth}.
We fix a basis~$\mathcal{B}$ of~$\fg$ as in \cref{thm:blattner}, so that, in particular, $b_1 < \dotsb < b_d$ is a basis of a complement of~$\fh$ in~$\fg$. 
Implicitly, when we write products of elements of $\fg$, we mean products in $U(\fg)$. 
For $u \in U(\fg)$, we write $\chi_i(u)$ to denote the component along $b_i$ of the expansion of $u$ in the Poincaré--Birkhoff--Witt basis associated with $\mathcal{B}$.
In particular, we will use the following elementary remark.

\begin{lemma} \label{lem:chi}
    For $n \in \N$, $a_1 \leq \dotsb \leq a_n \in \{ b_1, \dotsc, b_d \}$ and $g \in \fg$,
    \begin{equation}
        \chi_i(g a_n \dotsb a_1) = \sum_{k=1}^d \chi_k(g) \chi_i(b_k a_n \dotsb a_1).
    \end{equation}
\end{lemma}

\begin{proof}
    Since $\mathcal{B}$ is a (totally ordered) basis of $\fg$, there exists $s \in \N$, $b_{d+1} < \dotsb < b_{d+s} \in \mathcal{B}$ with $b_d < b_{d+1}$ and coefficients $\lambda_{d+1}, \dotsc, \lambda_{d+s} \in \K$ such that
    \begin{equation}
        g = \sum_{k=1}^d \chi_k(g) b_k + \sum_{k=d+1}^{d+s} \lambda_k b_k.
    \end{equation}
    For each $k \in \intset{d+1,d+s}$, $b_k > b_d \geq a_n$ so $b_k a_n \dotsb a_1$ belongs to the Poincaré--Birkhoff--Witt basis associated with $\mathcal{B}$ and hence $\chi_i(b_k a_n \dotsb a_1) = 0$.
    The formula follows by linearity of $\chi_i$.
\end{proof}

From now on, we use the following norm on the finite-dimensional quotient vector space $\fg / \fh$:
\begin{equation}
    \label{eq:def-norm-gh}
    \forall g \in \fg, \quad
    \| \chi(g) \|_{\fg / \fh} := \sum_{i=1}^d |\chi_i(g)|.
\end{equation}
We start with the following analytic-type estimate on quantities similar to the ones involved in formula \eqref{eq:formula-blattner}, but where $g$ is on the left.

\begin{proposition} \label{prop:left-estimate}
    There exists $C > 0$ such that, for every $n \in \N$ and $a_1 \leq \dotsb \leq a_n \in \{ b_1, \dotsc, b_d \}$, $g \in \fg$ and $i \in \intset{1, d}$,
    \begin{equation} \label{eq:left-estimate}
        | \chi_i(g a_n \dotsb a_1) | \leq C^n n! \| \chi(g) \|_{\fg / \fh}.
    \end{equation}
\end{proposition}

\begin{proof}
    By \cref{lem:chi}, it suffices to prove the estimate when $g \in \{ b_1, \dotsc, b_d \}$, and, in this case, $\| \chi(g) \|_{\fg / \fh} = 1$ by \eqref{eq:def-norm-gh}.

    \bigskip

    Let $F := \{ b_1, \dotsc, b_d \}$. 
    Let $C_F \ge 1$ be given by assumption \eqref{eq:growth}.
    We will prove \eqref{eq:left-estimate} with $C :=  2 C_F^2$.
    We proceed by induction on $n \in \N$. 
    If $n=0$, the conclusion holds by definition of $\| \cdot \|_{\fg / \fh}$ in \eqref{eq:def-norm-gh}. 
    Let $n \in \N^*$. 
    We assume that the estimate is proved up to $(n-1)$. 
    Let $i \in \intset{1,d}$, $g$ and $a_1 \leq \dotsb \leq a_n \in F$.
    Let $\nu \in \intset{0,n}$ be defined by
    \begin{equation}
        \nu :=
        \begin{cases}
            \max\{ j \in \intset{1,n} ; g > a_j \} & \text{ if } g>a_1\\
            0 & \text{ otherwise } 
        \end{cases}
    \end{equation}
    Then
    \begin{equation}
        g a_n \dotsb a_1 = a_n \dotsb a_{\nu+1} g a_{\nu} \dotsb a_1 + [g,a_n \dotsb a_{\nu+1}] a_{\nu}\dotsb a_1. 
    \end{equation}
    The first term on the right hand side is mapped to $0$ by $\chi_i$ and, by \cref{lem:bubble},
    \begin{equation}
        [g,a_n \dots a_{\nu+1}] = - \sum_{J \subsetneq \intset{\nu+1,n}} \mu_J(g) a_{j_h} \dots a_{j_1}.
    \end{equation}
    Thus, by \cref{lem:chi},
    \begin{equation}
        \chi_i \left( g a_n \dots a_1 \right) 
        = - \sum_{J \subsetneq \intset{\nu+1,n}}
        \sum_{k=1}^d \chi_k(\mu_J(g)) \chi_i \left( b_k a_{j_h} \dots a_{j_1} a_{\nu} \dots a_1 \right).
    \end{equation}
    For $J \subsetneq \intset{\nu+1,n}$, $\mu_J(g)$ is an iterated bracket of $(n-\nu-|J|+1)$ elements of $F$.
    Thus, by \eqref{eq:growth},
    \begin{equation}
         \sum_{k=1}^d |\chi_k(\mu_J(g))|
         =  \| \chi(\mu_J(g)) \|_{\fg / \fh}
         \leq C_F^{n-\nu-|J|+1} (n-\nu-|J|)!.
    \end{equation}
    Hence, using the induction assumption, and partitioning by the value of $h=|J|$,
    \begin{equation}
        \begin{split}
            |\chi_i \left( g a_n \dotsb a_1 \right)|
            & \leq \sum_{J \subsetneq \intset{\nu+1,n}}
            C_F^{n-\nu-|J|+1} (n-\nu-|J|)! C^{|J|+\nu} (|J|+\nu)! 
            \\ & \leq 
            \sum_{h=0}^{n-\nu-1} \binom{n-\nu}{h} C_F^{n-\nu-h+1} (n-\nu-h)! C^{h+\nu} (h+\nu)!
            \\ & \leq
            n! C_F C^{n} \sum_{h=0}^{n-\nu-1} \binom{n-\nu}{n-\nu-h} \binom{n}{n-\nu-h}^{-1} \left(\frac{C_F}{C}\right)^{n-\nu-h} 
            \\ & \leq 
            n! C_F C^n \sum_{\sigma = 1}^{+\infty} \left(\frac{C_F}{C}\right)^{\sigma} 
            \leq 2 C_F^2 C^{n-1} n! = C^n n!
        \end{split}
    \end{equation}
    because $\binom{n-\nu}{\cdot} \leq \binom{n}{\cdot}$, $C_F/C \leq 1/2$ and $C = 2 C_F^2$.
\end{proof}

We deduce the main analytic-type estimate on the quantities involved in formula \eqref{eq:formula-blattner}.

\begin{proposition} \label{prop:bound-coeff}
    For any $g \in \fg$, there exists $C_g>0$ such that, for every $m \in \N^d$ and $i \in \intset{1,d}$,
    \begin{equation}
        | \chi_i(\mathbf{b}^m g) | \leq C_g^{|m|+1} (|m|+1)!.
    \end{equation}
\end{proposition}

\begin{proof}
    Let $g \in \fg$.
    Let $F := \{ g \} \cup \{b_1,\dotsc,b_d\}$ and $C_F$ given by \eqref{eq:growth}.
    Let $C > 0$ be the constant given by \cref{prop:left-estimate}.
    Let $i \in \intset{1,d}$, $m \in \N^d$. 
    Let $n=|m|$ and $a_1 \leq \dotsb \leq a_n \in \{b_1,\dotsc,b_d\}$ such that $\mathbf{b}^m=a_n \dotsb a_1$.
    By \cref{lem:bubble} and linearity,
    \begin{equation}
        \chi_i (\mathbf{b}^m g) = \sum_{J \subset \intset{1,n}} \chi_i (\mu_J(g) a_{j_h} \dotsb a_{j_1}).
    \end{equation}
    By \cref{prop:left-estimate} and \eqref{eq:growth}, and partitioning by the value of $h=|J|$,
    \begin{equation}
        \begin{split}
            |\chi_i(\mathbf{b}^m g)| & \leq \sum_{J \subset \intset{1,n}} C^{|J|} |J|! \| \chi(\mu_J(g)) \|_{\fg / \fh} \\
            & \leq \sum_{J \subset \intset{1,n}} C^{|J|} |J|! C_F^{n-|J|+1} (n-|J|)! \\
            & \leq \sum_{h=0}^n \binom{n}{h} C^h h! C_F^{n-h+1} (n-h)!
            \leq (n+1)! C_g^{n+1}
        \end{split}
    \end{equation}
    with $C_g := \max \{ C, C_F \}$.
\end{proof}

Finally, this leads to our main result.

\begin{proof}[Proof of the reverse implication of \cref{thm:main}]
    Let $(\fg,\fh)$ be a transitive pair satisfying~\eqref{eq:growth}.
    By \cref{thm:blattner}, formula \eqref{eq:formula-blattner} defines a realization of $(\fg,\fh)$.
    Let $g \in \fg$.
    By \cref{prop:bound-coeff}, there exists $C_g > 0$ such that, for every $m \in \N^d$ and $i \in \intset{1,d}$,
    \begin{equation}
        \frac{|\chi_i(\mathbf{b}^m g)|}{m!} \leq C_g^{|m|+1} \frac{(m_1+\dotsb+m_d+1)!}{m_1! \dotsb m_d! 1!}
        \leq C^{|m|+1}
    \end{equation}
    where $C := (d+1) C_g$, by bounding the multinomial coefficient.
    Hence, $\rho(g)$ of \eqref{eq:formula-blattner} is a convergent formal vector field according to the definition of \cref{sec:intro-powers}.
    Thus, the realization $\rho$ given by \cref{thm:blattner} is a convergent realization of $(\fg,\fh)$.
\end{proof}

\subsection{Proof of the necessity of the convergence criterion}
\label{sec:necessary}

In this subsection, we prove that condition \eqref{eq:growth} is necessary for the existence of a convergent realization.
We start with classical definitions and analytic-type estimates for formal vector fields.
The following exposition is inspired by our previous work \cite[Section 3]{BeauchardLeBorgneMarbach2023} in the context of real-analytic vector fields.

\begin{definition}[Analytic norms]
    Given $r > 0$ and a formal power series $f \in \kx$, we define
    \begin{equation}
        \opnorm{f}_r := \sum_{m \in \N^d} r^{|m|} |f_m|.
    \end{equation}
    Given $r > 0$ and a formal vector field $v \in \Der \kx$, we define
    \begin{equation} \label{eq:def.analytic.r}
        \opnorm{v}_r := \sum_{i=1}^d \opnorm{v^i}_r.
    \end{equation}
    We will denote by $(\kx)_r$ and $(\Der \kx)_r$ the vector subspaces of $\kx$ and $\Der \kx$ on which these norms are finite.
    By definition, for every convergent formal power series $f$, there exists $r > 0$ such that $f \in (\kx)_r$, and for every convergent formal vector field $v$, there exists $r > 0$ such that $v \in (\Der \kx)_r$.
\end{definition}

\begin{lemma}[Submultiplicativity]
    \label{lem:submultiplicativity}
    Given $r > 0$ and formal power series $f, g \in (\kx)_r$,
	\begin{equation} \label{opnorm(fg)r}
		\opnorm{f g}_r \leq \opnorm{f}_r \opnorm{g}_r.
	\end{equation}
\end{lemma}

\begin{proof}
    For every $m \in \N^d$, $(fg)_m = \sum f_{m'} g_{m''}$ where the sum ranges over $m',m'' \in \N^d$ such that $m'+m''=m$. 
    Hence
    \begin{equation}
        \opnorm{fg}_r =
        \sum_{m \in \N^d} r^{|m|} | (fg)_m |
        \leq \sum_{m \in \N^d} \sum_{m' + m'' = m} r^{|m'|+|m''|} |f_{m'}| |g_{m''}|
        = \opnorm{f}_r \opnorm{g}_r,
    \end{equation}
    by the triangle inequality.
\end{proof}

\begin{lemma}[Control of gradients] 
    Given $0 < r_1 < r_2$, $f \in (\kx)_{r_2}$ and $i \in \intset{1,d}$,
	\begin{equation} \label{eq:djf.r.dr}
	    \opnorm{\partial_i f}_{r_1} \leq \frac{1}{r_2-r_1} \opnorm{f}_{r_2}.
	\end{equation}
\end{lemma}

\begin{proof}
	For every $m \in \N^d$, $(\partial_i f)_{m} = (m_i + 1) f_{m+e_i}$.
    Thus
	\begin{equation}
	\begin{split}
		\opnorm{\partial_i f}_{r_1}
		= \sum_{m \in \N^d} r_1^{|m|} (m_i+1) |f_{m+e_i}|
		& = \sum_{m \in \N^d} \frac{r_1^{|m|}}{r_2^{|m|+1}} (m_i+1) r_2^{|m+e_i|} |f_{m+e_i}| \\
		& \leq \opnorm{f}_{r_2} \sup_{m \in \N^d} \frac{r_1^{|m|}}{r_2^{|m|+1}} (m_i+1) \\
		& = \opnorm{f}_{r_2} \sup_{n \geq 1} n \frac{r_1^{n-1}}{r_2^n}.
	\end{split}
	\end{equation}
    By the binomial formula, $n r_1^{n-1} (r_2-r_1) \leq (r_1 + (r_2-r_1))^n = r_2^n$, which concludes the proof. 
\end{proof}

\begin{lemma}[Product estimate] 
    \label{thm:prod.analytic}
    Given $0 < r_1 < r_2$, $n \in \N$, $v_1, \dotsc v_n \in (\Der \kx)_{r_2}$ and $f \in (\kx)_{r_2}$,
    \begin{equation} \label{eq:prod.analytic}
        \opnorm{v_n \dotsb v_1 f}_{r_1}
        \leq n! \left(\frac{e}{r_2-r_1}\right)^n \opnorm{v_n}_{r_2} \dotsb \opnorm{v_1}_{r_2} \opnorm{f}_{r_2}.
    \end{equation}
\end{lemma}

\begin{proof}
    For $n = 0$, one uses that the norm \eqref{eq:def.analytic.r} is non-decreasing with respect to $r$.
	For $n = 1$, using \eqref{eq:def.analytic.r}, \eqref{opnorm(fg)r} and \eqref{eq:djf.r.dr}, we have
    \begin{equation}
        \label{eq:prodn1}
        \opnorm{v_1 f}_{r_1} \leq \frac{1}{r_2-r_1} \opnorm{v_1}_{r_1} \opnorm{f}_{r_2}.
    \end{equation}
    For $n > 1$, one applies  $n$ times \eqref{eq:prodn1} with a radius increment $(r_2-r_1) / n$ at each step. 
    This yields
	\begin{equation}
		\begin{split}
		    \opnorm{v_n \dotsb v_1 f}_{r_1}
		    & \leq 
			\left(\frac{n}{r_2-r_1}\right) 
			\opnorm{v_n}_{r_1} \opnorm{v_{n-1} \dotsb v_1 f}_{r_1+\frac{r_2-r_1}{n}} \\
			& \leq 
			\left(\frac{n}{r_2-r_1}\right)^n \opnorm{f}_{r_2}
			\prod_{j=1}^n \opnorm{v_j}_{r_1+(n-j)\frac{r_2-r_1}{n}},
		\end{split}
	\end{equation}
	which concludes the proof by monotonicity of \eqref{eq:def.analytic.r}, and using $n^n \leq n! e^n$.
\end{proof}

\begin{lemma}[Analytic estimate] 
    \label{thm:bracket.analytic}
    Given $0 < r_1 < r_2$, $n \in \N^*$ and $v_1, \dotsc v_n \in (\Der \kx)_{r_2}$,
    \begin{equation}
        \opnorm{[v_1, [v_2, [\dotsc, [v_{n-1},v_n] \dotsb ]]]}_{r_1} 
        \leq (n-1)! \left( \frac{2e}{r_2-r_1} \right)^{n-1}
        \prod_{k = 1}^n \opnorm{v_k}_{r_2}.
    \end{equation}
\end{lemma}

\begin{proof}
    This estimate stems from \eqref{eq:prod.analytic} because, as can be checked by induction on $n$, the iterated Lie bracket $[v_1, [v_2, [\dotsc, [v_{n-1},v_n] \dotsb ]]]$ is a sum of at most $2^{n-1}$ terms of the form studied in \cref{thm:prod.analytic}, where $f$ is a component of one of the vector fields $v_k$.
\end{proof}

\begin{proof}[Proof of the forward implication of \cref{thm:main}]
    Let $(\fg,\fh)$ be a transitive pair and $d' := \codim_\fg \fh$.
    Assume that it has a convergent realization $\rho : \fg \to \Der \kx$ in $d \geq d'$ variables.
    We want to prove that \eqref{eq:growth} holds.
    
    Let $b_1, \dotsc, b_{d'} \in \fg$ be a basis of a complement of $\fh$ in $\fg$.
    Then any element $g \in \fg$ has a unique decomposition of the form
    \begin{equation}
        g = h + \sum_{i=1}^{d'} \chi_i(g) b_i
    \end{equation}
    where $h \in \fh$ and $\chi_i(g) \in \K$.
    We use the norm \eqref{eq:def-norm-gh} on $\fg/\fh$.

    Since $\rho$ is a realization of $(\fg,\fh)$, for every $h \in \fh$, $\rho(h)(0) = 0$, and the vectors $\rho(b_i)(0)$ are linearly independent in $\K^d$.
    Let $V = \vect (\rho(b_1)(0), \dots, \rho(b_{d'})(0)) \subset \K^d$, and let $N_1$ be the norm on $V$ defined by $N_1(\sum \lambda_i \rho(b_i)(0)) = \sum |\lambda_i|$.
    Then, since $\dim V$ is finite, there exists a constant $C > 0$ such that $N_1 \leq C \| \cdot \|_{\K^d}$. 
    By \eqref{eq:def-norm-gh}, for every $g \in \fg$, 
    \begin{equation}
        \| \chi(g) \|_{\fg/\fh} = \sum_{i=1}^{d'} |\chi_i(g)| = N_1(\rho(g)(0)) \leq C \| \rho(g)(0) \|_{\K^d}.
    \end{equation}

    Let $F \subset \fg$ be a finite set.
    For every $g \in F$, $\rho(g)$ is a convergent formal vector field since $\rho$ is a convergent realization.
    Hence, since $F$ is finite, there exists $C_r, r > 0$ such that, for every $g \in F$, $\rho(g) \in (\Der \kx)_r$ and $\opnorm{\rho(g)}_r \leq C_r$.
    Let $n \in \N^*$ and $g_1, \dotsc, g_n \in F$.
    Then, using \cref{thm:bracket.analytic},
    \begin{equation}
        \begin{split}
            \| \chi([g_1, [g_2, [\dotsc, [g_{n-1},g_n]\dotsb]]]) \|_{\fg/\fh}
            & \leq C \| [\rho(g_1), [\rho(g_2), [\dotsc, [\rho(g_{n-1}),\rho(g_n)]\dotsb]]] (0) \|_{\K^d} \\
            & \leq C (n-1)! \left(\frac{2e}{r} \right)^{n-1} C_r^n
        \end{split}
    \end{equation}
    and this estimate corresponds to \eqref{eq:growth} with $C_F := \max \{ C C_r, \frac{2e C_r}{r} \}$.
\end{proof}

\section{Changes of coordinates and comparison of realizations}
\label{sec:coordinate-changes}

This section gathers change-of-coordinates and comparison results surrounding realizations, first for formal power series, then for convergent ones. 
More precisely, the setting is the following.

Let $\fh, \fh'$ be Lie subalgebras of a Lie algebra $\fg$ of codimension $d = \codim_\fg \fh$ and $d' = \codim_\fg \fh'$.
Let $\rho : \fg \to \Der \kx$ be a realization of $(\fg, \fh)$ and $\rho' : \fg \to \Der \ky$ be a realization of $(\fg,\fh')$, in minimal dimensions $d$ and $d'$.
We look for a local homomorphism $\tau : \ky \to \kx$ such that
\begin{equation}
    \label{eq:swap-rho-rho'}
    \forall g \in \fg, \quad \tau \circ \rho'(g) = \rho(g) \circ \tau.
\end{equation}
\emph{i.e.} such that $\tau$ is a homomorphism of representations of the Lie algebra $\fg$.
We prove that such a~$\tau$ is unique, exists if and only if $\fh \subset \fh'$, and is convergent when $\rho, \rho'$ are.

\subsection{Preliminary remarks}

\subsubsection{Local homomorphisms}

We recall that $\kx = \K \llbracket x_1, \dotsc, x_d \rrbracket$ is a \emph{local ring}: it admits a unique maximal ideal $\mathfrak{m}_x$, the subset of power series with a zero constant term.
Given another ring of power series $\ky = \K \llbracket y_1, \dotsc, y_{d'} \rrbracket$, a $\K$-algebra homomorphism $\tau : \ky \to \kx$ is called \emph{local} when $\tau(\mathfrak{m}_y) \subset \mathfrak{m}_x$.

\begin{lemma}
    \label{lem:auto-local}
    Any $\K$-algebra isomorphism $\ky \to \kx$ is local.
\end{lemma}

\begin{lemma}
    \label{lem:local-homo}
    Let $\tau : \ky \to \kx$ be a local homomorphism.
    For any $f \in \ky$, $(\tau f)(0) = f(0)$.
\end{lemma}

\begin{lemma}
    \label{lem:local-homo-substitution}
    Let $\tau:\ky\to\kx$ be a local homomorphism. 
    Then, for every $f\in\ky$, one has $\tau(f)=f\bigl(\tau(y_1),\dotsc,\tau(y_{d'})\bigr)$.
    Conversely, any $Y_1,\dotsc,Y_{d'}\in\mathfrak m_x$ uniquely determine a local homomorphism $\tau:\ky\to\kx$ such that $\tau(y_i)=Y_i$.
\end{lemma}

\subsubsection{Extension of realizations to the universal enveloping algebra}

Let $\rho : \fg \to \Der \kx$ be a realization of a transitive pair $(\fg, \fh)$ in minimal dimension.
By the universal property of $U(\fg)$, it admits a unique extension $\rho : U(\fg) \to \End \kx$ such that, for any $g_1, \dotsc, g_n \in \fg$, one has $\rho(g_1 \dotsb g_n) = \rho(g_1) \dotsb \rho(g_n)$.

\begin{lemma}
    \label{lem:transitive-realization-separates-series}
    Let $\rho : \fg \to \Der \kx$ be a realization of a transitive pair $(\fg, \fh)$ with $d = \codim_\fg \fh$.
    Let $f\in\kx$.
    Then $f = 0$ if and only if
    \begin{equation}
        \label{eq:rho-F=0}
        \forall a \in U(\fg),
        \qquad
        (\rho(a) f)(0)=0.
    \end{equation}
\end{lemma}

\begin{proof}
    Assume that \eqref{eq:rho-F=0} holds.
    Choose $b_1,\dotsc,b_d\in\fg$ whose classes form a basis of $\fg/\fh$.
    Since $\rho$ is a realization of $(\fg,\fh)$, the vectors $\rho(b_1)(0),\dotsc,\rho(b_d)(0)$ form a basis of $\K^d$.

    We prove by induction on $n\in\N$ that all homogeneous terms of $f$ of degree $n$ vanish.
    For $n = 0$, the assertion follows from the assumption with $a = 1$: $f_0 = f(0) = (\rho(1) f)(0) = 0$.

    Assume that all homogeneous terms of $f$ of degree $<n$ vanish. Let
    $i_1,\dotsc,i_n\in\llbracket 1,d\rrbracket$. Then
    \begin{equation}
        \big(\rho(b_{i_1})\dotsm\rho(b_{i_n})f\big)(0) = 0.
    \end{equation}
    By the induction hypothesis, all terms involving derivatives of the
    coefficients of the vector fields only depend on derivatives of $f$ of order $<n$, hence vanish at the origin.
    Therefore the previous identity reduces to
    \begin{equation}
        \mathrm d^n f(0)
        \big(
        \rho(b_{i_1})(0),
        \dotsc,
        \rho(b_{i_n})(0)
        \big)
        =
        0.
    \end{equation}
    Since $\rho(b_1)(0),\dotsc,\rho(b_d)(0)$ form a basis of $\K^d$, this implies that $\mathrm d^n f(0)=0$.
    Thus the homogeneous term of degree $n$ vanishes.
    The induction proves $f=0$.
\end{proof}

\subsection{Formal conjugation homomorphisms}
\label{sec:formal-conjugation}

We first prove that local homomorphisms satisfying the conjugation identity \eqref{eq:swap-rho-rho'} are unique.

\begin{proposition}
    \label{prop:unique-equivariant-formal-morphism}
    Let $\fh, \fh' \subset\fg$ be Lie subalgebras of codimension $d = \codim_\fg \fh$ and $d' = \codim_\fg \fh'$.
    Let $\rho :\fg \to \Der \kx$ and $\rho' : \fg \to \Der \ky$ be realizations of $(\fg,\fh)$ and $(\fg,\fh')$. 
    
    There is at most one local homomorphism $\tau : \ky \to \kx$ such that \eqref{eq:swap-rho-rho'} holds.
\end{proposition}

\begin{proof}
    Let $\tau_1,\tau_2$ be two such homomorphisms.
    For $i\in\llbracket 1,d'\rrbracket$, set $f_i := \tau_1(y_i) - \tau_2(y_i) \in \kx$.
    The $\fg$-equivariance relation \eqref{eq:swap-rho-rho'} extends naturally to a $U(\fg)$-equivariance relation.
    Thus, for every $a\in U(\fg)$, using \cref{lem:local-homo},
    \begin{equation}
        \begin{split}
            (\rho(a) f_i)(0)
            &= \bigl(\tau_1(\rho'(a) y_i)\bigr)(0) - \bigl(\tau_2(\rho'(a) y_i)\bigr)(0)
            =0.
        \end{split}
    \end{equation}
    By \cref{lem:transitive-realization-separates-series}, we get $f_i=0$ for all $i$, hence $\tau_1=\tau_2$ by \cref{lem:local-homo-substitution}.
\end{proof}

We now prove that a local homomorphism satisfying the conjugation identity \eqref{eq:swap-rho-rho'} exists if and only if $\fh \subset \fh'$.
We start with the case where $\rho$ is a canonical Blattner--Draisma realization.

\begin{proposition}
    \label{prop:canonical-comparison-morphism}
    Let $\fh, \fh' \subset\fg$ be Lie subalgebras of codimension $d = \codim_\fg \fh$ and $d' = \codim_\fg \fh'$.
    Let $\rho :\fg \to \Der \kx$ and $\rho' : \fg \to \Der \ky$ be realizations of $(\fg,\fh)$ and $(\fg,\fh')$. 

    Assume that $\fh \subset \fh'$ and that $\rho$ is the realization given by \cref{thm:blattner} for a given tuple $b_1, \dotsc, b_d$ whose classes form a basis of $\fg / \fh$.
    Recall $\mathcal{E}(x)$ of \eqref{eq:gen-E} and define, for $f \in \ky$,
    \begin{equation}
        \label{eq:Psi}
        \Psi(f) := \big( \rho'(\mathcal{E}(x)) f \big)\vert_{y = 0}.
    \end{equation}
    Then $\Psi$ is a local homomorphism from $\ky$ to $\kx$ satisfying \eqref{eq:swap-rho-rho'}.

    Moreover, if $\fh = \fh'$, then $\Psi$ is an isomorphism.
\end{proposition}

\begin{proof}
    The exponentials $e^{x_i \rho'(b_i)}$ are exponentials of derivations, hence algebra automorphisms after extension of scalars to $\K\llbracket x,y\rrbracket$. 
    Evaluation at $y=0$ is an algebra homomorphism. 
    Thus $\Psi$ is a $\K$-algebra homomorphism, and $\Psi(f)(0)=f(0)$, so it is local.

    Let $\chi_i$ be the coefficient functionals associated with the chosen basis of $\fg/\fh$. 
    Let $g \in \fg$.
    Recalling the PBW congruence \eqref{eq:key-congruence}, we have:
    \begin{equation}
        \mathcal{E}(x) g
        =
        \sum_{j=1}^{d} \chi_j(\mathcal{E}(x)g)\partial_{x_j}\mathcal{E}(x)
        + R_g(x)
        \qquad \text{where } \quad
        R_g(x)\in \fh U(\fg) \lxl.
    \end{equation}
    Since $\rho'$ is a realization of $(\fg,\fh')$ and $\fh \subset \fh'$, for any $h \in \fh$ and $f \in \ky$, $(\rho'(h) f)\vert_{y=0} = 0$.
    Therefore, for any $f \in \ky$, $(\rho'(R_g(x)) f) \vert_{y=0} = 0$.
    Thus, for any $f \in \ky$,
    \begin{equation}
        \begin{split}
            \Psi(\rho'(g)f)
            &=
            \left.\bigl(\rho'(\mathcal{E}(x)g)f\bigr)\right|_{y=0} \\
            &=
            \sum_{j=1}^{d}
            \chi_j (\mathcal{E}(x)g)
            \partial_{x_j}
            \left.\bigl(\rho'(\mathcal{E}(x))f\bigr)\right|_{y=0} \\
            &=
            \rho(g)\Psi(f),
        \end{split}
    \end{equation}
    which proves that $\Psi$ satisfies \eqref{eq:swap-rho-rho'}.

    If $\fh=\fh'$, then $d=d'$.
    For $1 \leq j \leq d$ and $f \in \ky$, one has $(\partial_{x_j} \Psi f)\vert_{x=0} = \rho'(b_j) f \vert_{y = 0}$.
    Thus the columns of the matrix $(\partial_{x_j} (\Psi y_i) \vert_{x=0})_{1 \leq i, j \leq d}$ are the $\rho'(b_1)(0), \dotsc, \rho'(b_d)(0)$.
    These form a basis of $\K^d$ because the classes of $b_1,\dotsc,b_d$ form a basis of $\fg/\fh $ and $\rho'$ realizes $(\fg,\fh)$ since $\fh = \fh'$. 
    Hence $\Psi$ is an isomorphism by the inverse function theorem for formal power series (see \cref{prop:ift-formal} in \cref{sec:ift-formal}).
\end{proof}

We now turn to the general case.

\begin{proposition}
    \label{thm:formal-morphism-transitive-pairs}
    Let $\fh, \fh' \subset\fg$ be Lie subalgebras of codimension $d = \codim_\fg \fh$ and $d' = \codim_\fg \fh'$.
    Let $\rho :\fg \to \Der \kx$ and $\rho' : \fg \to \Der \ky$ be realizations of $(\fg,\fh)$ and $(\fg,\fh')$. 

    There exists a local homomorphism $\tau : \ky \to \kx$ such that \eqref{eq:swap-rho-rho'} holds if and only if $\fh \subset \fh'$.
    When it exists, it is unique.
    When $\fh = \fh'$, it is an isomorphism.
\end{proposition}

\begin{proof}
    Choose $b_1,\dotsc,b_d$ whose classes form a basis of $\fg/\fh$, and let $\bar{\rho}$ be the corresponding realization given by \cref{thm:blattner}.
    By \cref{prop:canonical-comparison-morphism}, applied to $\rho'$ and $\bar{\rho}$, there exists a local homomorphism $\Psi' : \ky \to \kx$ such that, for all $g \in \fg$, $\Psi' \circ \rho'(g) = \bar{\rho}(g) \circ \Psi'$.
    By \cref{prop:canonical-comparison-morphism}, applied to $\rho$ and $\bar{\rho}$ (in the equality case), there exists an automorphism $\Psi$ of $\kx$ such that, for all $g \in \fg$, $\Psi \circ \rho(g) = \bar{\rho}(g) \circ \Psi$.
    Then $\tau := \Psi^{-1} \circ \Psi' : \ky \to \kx$ is a local homomorphism which satisfies~\eqref{eq:swap-rho-rho'}.
    Uniqueness was proved in \cref{prop:unique-equivariant-formal-morphism}.
    When $\fh = \fh'$, $\Psi'$ was also an isomorphism, so $\tau$ too.
    
    Conversely, assume such a $\tau$ exists and let $g\in\fh$. 
    Then $\rho(g)(0) = 0$.
    For any $f \in \ky$, by \cref{lem:local-homo},
    \begin{equation}
        (\rho'(g) f)(0) 
        = \bigl(\tau(\rho'(g)f)\bigr)(0)
        = \bigl(\rho(g)\tau(f)\bigr)(0)
        = 0.
    \end{equation}
    Hence $\rho'(g)\in\Der_0\ky$ and $g\in\fh'$, so $\fh \subset \fh'$.
\end{proof}

\cref{thm:formal-morphism-transitive-pairs} entails the following relation between two realizations of a given transitive pair.

\begin{corollary}
    \label{thm:uniqueness}
    Let $\rho_1, \rho_2 : \fg \to \Der \kx$ be two realizations of a transitive pair $(\fg,\fh)$ in minimal dimension.
    Then there exists a unique automorphism $\tau$ of $\kx$ such that 
    \begin{equation}
        \label{eq:tau-rho}
        \forall g \in \fg, \quad \tau \circ \rho_2(g) = \rho_1(g) \circ \tau.
    \end{equation}
\end{corollary}

Finally, we note that conjugation allows to produce other realizations from a given one.

\begin{lemma}
    \label{lem:conjugation-is-real-formal}
    Let $\rho' : \fg \to \Der \ky$ be a realization of a transitive pair $(\fg,\fh)$ in minimal dimension.
    Let $\tau : \ky \to \kx$ be an isomorphism.
    Then $\tau_* \rho' : g \mapsto \tau \circ \rho'(g) \circ \tau^{-1}$ is a realization of $(\fg,\fh)$.
\end{lemma}

\begin{proof}
    Since $\tau$ is an isomorphism, $\tau_* \rho' : \fg \to \Der \kx$ is an homomorphism of Lie algebras.
    By \cref{lem:auto-local,lem:local-homo}, for every $f\in\ky$,
    \begin{equation}
        ((\tau_*\rho'(g))(\tau f))(0)=(\rho'(g)f)(0).
    \end{equation}
    Since $\tau$ is an isomorphism, the linear parts of the $\tau y_i$ form a basis of $\mathfrak m_x/\mathfrak m_x^2$. 
    Hence $\tau_*\rho'(g)$ vanishes at $0$ if and only if $\rho'(g)$ vanishes at $0$.
    So $\fh = (\tau_* \rho')^{-1}(\Der_0 \kx)$.
\end{proof}

\subsection{Exhaustiveness of the realization formula}

By \cref{thm:uniqueness}, given any $b_1, \dotsc, b_d$ as in \cref{thm:blattner}, any other realization $\rho'$ is conjugated to the realization $\rho$ given by formula~\eqref{eq:formula-blattner} through a change of coordinates.
We wondered whether, for any realization $\rho'$, there exists a choice of $b_1, \dotsc, b_d$ such that $\rho'$ is given by~\eqref{eq:formula-blattner} for this choice of $b_1, \dotsc, b_d$.
This is false, as illustrated by the following counterexample.

\begin{example}
    Let $\fg$ denote the one-dimensional (commutative) Lie algebra over $\K$ and $\fh = \{0\}$ denote the zero Lie algebra (of codimension 1 in $\fg$, so $d = 1$).
    We write $\fg = \K g_1$ for some fixed generator $g_1 \in \fg$.
    To apply \cref{thm:blattner}, we need to choose $b_1$ spanning a complement of $\fh$ in $\fg$, so $b_1 = c g_1$ for some $c \neq 0$.
    Hence
    \begin{equation}
        \rho(g_1) = \sum_{m_1 \in \N} \frac{\chi_1(b_1^{m_1} g_1)}{m_1!} x_1^{m_1} \partial_1
        = \frac 1 c \sum_{m_1 \in \N} \frac{\chi_1(b_1^{m_1+1})}{m_1!} x_1^{m_1} \partial_1
        = \frac 1 c \partial_1
    \end{equation}
    since $\chi_1(b_1^k) = \mathbf{1}_{k = 1}$.
    Thus, in this setting, all realizations of the Blattner--Draisma form \cref{eq:formula-blattner} are proportional to the canonical realization $g_1 \mapsto \partial_1$.
    However, many other realizations exist.
    Indeed, given any formal power series $f \in \K \llbracket x_1 \rrbracket$ with $f_0 \neq 0$, $\rho(\lambda g_1) := \lambda f(x_1) \partial_1$ defines a realization of~$(\fg,\fh)$.
    In fact, all realizations of $(\fg,\fh)$ are of this form.
\end{example}

\subsection{Convergent changes of coordinates}

We now turn to the convergent counterparts of the formal results of \cref{sec:formal-conjugation}.

\begin{definition}
    \label{def:conv-homo}
    We say that a local homomorphism $\tau : \ky \to \kx$ is \emph{convergent} when, for each $i \in \intset{1,d'}$, $\tau y_i \in \kx$ is a convergent power series.
\end{definition}

We will use the following associated estimates.

\begin{lemma}[Composition estimate]
    \label{lem:composition-estimate}
    Let $Y=(Y_1,\dotsc,Y_{d'})$ with $Y_i \in \mathfrak{m}_x$.
    Let $R,r>0$.
    Assume that $Y_i\in(\kx)_r$ and $\opnorm{Y_i}_r\leq R$ for all $i$.
    Then, for any $f\in(\ky)_R$, $f(Y(x)) \in (\kx)_r$ and
    \begin{equation}
        \opnorm{f(Y)}_r\leq \opnorm{f}_R.
    \end{equation}
\end{lemma}

\begin{proof}
    For any $m \in \N^{d'}$, submultiplicativity gives $\opnorm{Y^m}_r \leq R^{|m|}$ (see \cref{lem:submultiplicativity}). 
    Writing $f(Y)=\sum_m f_m Y^m$ and summing over $m$ gives the estimate.
\end{proof}

\begin{lemma}[Composition of convergent series]
    \label{cor:composition-convergent-series}
    Let $\tau:\ky\to\kx$ be a convergent local homomorphism. 
    If $f\in\ky$ is convergent, then $\tau f$ is convergent.
\end{lemma}

\begin{proof}
    Choose $R>0$ such that $f \in (\ky)_R$. 
    Since the $\tau y_i$ are convergent and have zero constant term, one may choose $r>0$ such that $\tau y_i \in (\kx)_r$ and $\opnorm{\tau(y_i)}_r \leq R$.

    By \cref{lem:local-homo-substitution}, $\tau f = f(\tau y_1, \dotsc, \tau y_{d'})$, so the conclusion follows from \cref{lem:composition-estimate}.
\end{proof}

By the inverse function theorem for convergent power series (see \cref{sec:ift-convergent}), when $\tau$ is an isomorphism, $\tau$ is convergent if and only if $\tau^{-1}$ is.

Thus, conjugating by a convergent isomorphism preserves convergence.

\begin{lemma}
    \label{lem:conjugation-is-real-convergent}
    Let $\rho' : \fg \to \Der \ky$ be a convergent realization of a transitive pair $(\fg,\fh)$ in minimal dimension.
    Let $\tau : \ky \to \kx$ be a convergent isomorphism.
    
    Then $\tau_* \rho' : g \mapsto \tau \circ \rho'(g) \circ \tau^{-1}$ is a convergent realization of $(\fg,\fh)$.
\end{lemma}

\begin{proof}
    By \cref{lem:conjugation-is-real-formal}, $\tau_* \rho'$ is a formal realization of $(\fg,\fh)$.
    It remains to prove that, for any $g \in \fg$, $(\tau_* \rho')(g)$ is convergent.
    Let $g \in \fg$. 
    By assumption $\rho'(g)$ is a convergent derivation.
    By \cref{prop:ift-convergent}, $\tau^{-1}$ is convergent.
    By \cref{cor:composition-convergent-series} and \cref{thm:prod.analytic}, for all $i$, $(\tau_* \rho'(g)) x_i = \tau \circ \rho'(g) \circ \tau^{-1}(x_i)$ is convergent. 
\end{proof}

We now turn to the main results of this section.

\begin{proposition}
    \label{prop:canonical-comparison-morphism-convergent}
    In the setting of \cref{prop:canonical-comparison-morphism}, assume that $\rho'$ is convergent. 
    Then the canonical comparison homomorphism $\Psi$ given by \eqref{eq:Psi} is convergent.
\end{proposition}

\begin{proof}
    For each coordinate $y_i$,
    \begin{equation}
        \Psi(y_i) = \sum_{m\in\N^d} \frac{\bigl(\rho'(b_d^{m_d}\dotsm b_1^{m_1}) y_i\bigr)(0)}{m!} x^m .
    \end{equation}
    The finite family $\rho'(b_1),\dotsc,\rho'(b_d)$ is bounded in $(\Der \ky)_r$ for some $r>0$. 
    Applying \cref{thm:prod.analytic} as in the proof of the sufficiency estimate gives a constant $C>0$ such that the coefficient of $x^m$ in $\Psi(y_i)$ has absolute value at most $C^{|m|}$. 
    Hence each $\Psi(y_i)$ is convergent.
\end{proof}

As for formal realizations, the general case follows.

\begin{proposition}
    \label{thm:convergent-morphism-transitive-pairs}
    Let $\fh, \fh' \subset\fg$ be Lie subalgebras of finite codimension.
    Let $\rho :\fg \to \Der \kx$ and $\rho' : \fg \to \Der \ky$ be  convergent realizations of $(\fg,\fh)$ and $(\fg,\fh')$ in minimal dimension.
    
    There exists a unique convergent local homomorphism $\tau : \ky \to \kx$ such that \eqref{eq:swap-rho-rho'} holds if and only if $\fh \subset \fh'$.
    When $\fh = \fh'$, it is an isomorphism and $\tau^{-1}$ is convergent.
\end{proposition}

\begin{proof}
    The proof is the same as \cref{thm:formal-morphism-transitive-pairs}.
    The homomorphisms $\Psi$ and $\Psi'$ are convergent by \cref{prop:canonical-comparison-morphism-convergent} and so is $\tau := \Psi^{-1} \circ \Psi'$ by \cref{prop:ift-convergent} and \cref{cor:composition-convergent-series}.
\end{proof}

Finally, we note that the growth condition \eqref{eq:growth} is monotone with respect to $\fh$.

\begin{lemma}
    \label{lem:growth-mono}
    Let $\fh \subset \fh' \subset \fg$ be Lie subalgebras, with $\fh$ of finite codimension.
    If $(\fg,\fh)$ satisfies the growth condition \eqref{eq:growth}, then
    $(\fg,\fh')$ satisfies \eqref{eq:growth} too.
\end{lemma}

\begin{proof}
    Let $\chi:\fg\to \fg/\fh$ and $\chi':\fg\to \fg/\fh'$ be the quotient maps.
    Since $\fh\subset\fh'$, there is a unique linear map $\pi : \fg/\fh \to \fg/\fh'$ such that $\chi'=\pi\circ\chi$.
    Since $\fg/\fh$ is of finite dimension, $\pi$ is bounded: there exists $A>0$ such that, for all $v \in \fg/\fh$, $\|\pi(v)\|_{\fg/\fh'}\leq A\|v\|_{\fg/\fh}$.
    Applying this to every iterated bracket appearing in \eqref{eq:growth}, we get
    \begin{equation}
        \|\chi'([g_1,[\dotsc,[g_{n-1},g_n]\dotsb]])\|_{\fg/\fh'}
        \leq
        A\|\chi([g_1,[\dotsc,[g_{n-1},g_n]\dotsb]])\|_{\fg/\fh}.
    \end{equation}
    The extra factor $A$ is absorbed into the exponential constant. 
    Hence $(\fg,\fh')$ satisfies \eqref{eq:growth}.
\end{proof}

\section{Output realizations}
\label{sec:output-realizations}

Let $\fg$ be a Lie algebra over $\K$.
In this section, we study which linear subspaces $\fp$ of $U(\fg)$ can occur as kernels of finite-dimensional outputs attached to formal realizations.
We characterize such kernels first at the formal level, then discuss the uniqueness and convergence of the associated realizations.
At the formal level, this question is also solved in \cite{GrossmanLarson1992} (see \cref{rk:larson}).

\subsection{Output kernels and the canonical formula}

Let $d, q \in \N^*$.
We write $x = (x_1, \dotsc, x_d)$ and $\K^q \lxl$ for formal series with coefficients in $\K^q$.
Given a homomorphism of Lie algebras $\rho : \fg \to \Der \kx$, we still denote by $\rho : U(\fg)\to\End \kx$ its canonical extension to the universal enveloping algebra. 
It acts componentwise on $\K^q \lxl$.

We now fix a linear subspace $\fp$ of $U(\fg)$.
We write $\pi_\fp : U(\fg) \to U(\fg) / \fp$ for the quotient map.

\begin{definition}
    An \emph{output realization in dimensions $(d,q)$} of the pair
    $(\fg,\fp)$ is a pair $(\rho,h)$, where $\rho : \fg \to \Der \kx$ is a homomorphism of Lie algebras and $h \in \K^q \lxl$ with $h(0) = 0$ is such that
    \begin{equation}
        \label{eq:def-output-m-eq-ker}
        \fp
        =
        \left\{ a\in U(\fg) \mid (\rho(a)h)(0)=0\right\}.
    \end{equation}
\end{definition}

\begin{definition}
    The \emph{output isotropy} of $\fp$ is
    \begin{equation}
        \fhp
        :=
        \{g\in\fg \mid gU(\fg)\subset\fp\}.
    \end{equation}
\end{definition}

\begin{lemma}
    \label{lem:hm-lie-subalgebra}
    The output isotropy $\fhp$ is a Lie subalgebra of $\fg$.
\end{lemma}

The main result of this section is the following characterization and explicit construction.

\begin{theorem}
    \label{thm:output-realization}
    The pair $(\fg,\fp)$ admits an output realization if and only if
    \begin{equation}
        \label{eq:output-realization-condition}
        1\in\fp,
        \qquad
        \codim_{U(\fg)}\fp<+\infty,
        \qquad
        \codim_{\fg}\fhp<+\infty.
    \end{equation}
    More explicitly, when \eqref{eq:output-realization-condition} holds, set $d:=\codim_{\fg}\fhp$ and $q:=\codim_{U(\fg)}\fp$.
    Then identify $U(\fg) / \fp$ with $\K^q$.   
    Choose $b_1,\dotsc,b_d\in\fg$ whose classes form a basis of $\fg/\fhp$ and define
    \begin{equation}
        \label{eq:canonical-output-h}
        h(x) := \sum_{m\in\N^d}
        \pi_{\fp}(\mathbf b^m)\frac{x^m}{m!}
        \in \K^q \lxl.
    \end{equation}
    Then $(\rho,h)$ is an output realization of $(\fg,\fp)$ in dimensions $(d,q)$, where $\rho : \fg \to \Der \kx$ is the canonical realization of the transitive pair $(\fg,\fhp)$ given by \cref{thm:blattner}.
\end{theorem}

\begin{proof}
    \emph{Necessity.}
    Let $(\rho,h)$ be an output realization in dimensions $(d,q)$.
    
    Define a linear map $G : U(\fg) \to \K^q$ by $G(a) := (\rho(a) h)(0)$.
    By \eqref{eq:def-output-m-eq-ker}, $\fp = \ker G$.
    Hence $\codim_{U(\fg)} \fp \leq q$.
    Moreover, $G(1) = h(0) = 0$, so $1\in\fp$.

    Consider the linear map $E_\rho : \fg \to \K^d$ defined by $E_\rho(g) := \rho(g)(0)$.
    Given $g \in \ker E_\rho$, for any $f \in \kx$, $(\rho(g) f)(0) = \mathrm{d}f(0) \rho(g)(0) = 0$.
    Hence, for any $g \in \ker E_\rho$ and $a \in U(\fg)$, $G(ga) = (\rho(g) \rho(a) h)(0) = 0$.
    Thus $\ker E_\rho \subset \fhp$.
    Therefore $\codim_{\fg}\fhp \leq \codim_{\fg}\ker(E_\rho) \leq d$.

    \medskip \noindent
    \emph{Sufficiency}.
    Conversely, assume \eqref{eq:output-realization-condition}.
    We now prove that $(\rho,h)$ realizes $\fp$.
    Because $1\in\fp$, we have $h(0) = \pi_{\fp}(1) = 0$.
    Recalling the generating series $\mathcal{E}(x)$ from \eqref{eq:gen-E},
    set, for $a \in U(\fg)$,
    \begin{equation}
        F_a(x):=\pi_{\fp}(\mathcal{E}(x)a) \in \K^q \lxl.
    \end{equation}
    In particular, $F_1 = h$.
    Let $g\in\fg$ and $a \in U(\fg)$.

    Multiplying on the right by $a$ the generating-series identity
    \eqref{eq:key-congruence} obtained in the proof of \cref{thm:blattner},
    composing with $\pi_{\fp}$ and using that $\pi_{\fp}=0$ on
    $\fhp U(\fg)$, we obtain
    \begin{equation}
        \pi_{\fp}(\mathcal{E}(x) g a) = \sum_{j=1}^d \chi_j(\mathcal{E}(x) g) \partial_j (\pi_{\fp} (\mathcal{E}(x) a)).
    \end{equation}
    Therefore $F_{ga} = \rho(g) F_a$.
    By induction on products in $U(\fg)$, we obtain $F_a = \rho(a) h$ for all $a \in U(\fg)$.
    Evaluating at $x=0$, we get $(\rho(a) h)(0) = F_a(0) = \pi_{\fp}(a)$.
    Hence \eqref{eq:def-output-m-eq-ker} holds.
\end{proof}

\begin{remark}
    \label{rk:larson}
    \cref{thm:output-realization} is very closely related to \cite[Theorem 1.1]{GrossmanLarson1992}.
    In this work, the authors study the case $q = 1$ (scalar output).
    Starting from a primitively generated bialgebra $H$, they characterize the functionals $p \in H^*$ which are differentially produced by a finite-dimensional formal state space, by a finite Lie-rank condition. 
    When $H=U(\fg)$ and $\fp=\ker p$, this condition is, up to the left/right convention, the finite codimension of the output isotropy $\fh_\fp$.
    In particular, we share the same explicit formula \eqref{eq:canonical-output-h} (see \cite[Lemma 2.2]{GrossmanLarson1992}).
\end{remark}

\subsection{Uniqueness of minimal output realizations}
\label{sec:output-uniqueness}

We prove the natural uniqueness statement for output realizations in minimal dimensions. 

\begin{proposition}
    \label{prop:output-uniqueness-with-output-matrix}
    Let $\fp\subset U(\fg)$ satisfying \eqref{eq:output-realization-condition}.
    Set $d:=\codim_{\fg}\fhp$ and $q:=\codim_{U(\fg)}\fp$.
    Let $(\rho_1, h_1)$ and $(\rho_2, h_2)$ be two output realizations of $(\fg,\fp)$ in dimensions $(d,q)$.
    Then there exist an automorphism $\tau$ of $\kx$ satisfying \eqref{eq:tau-rho} and a matrix $A \in \mathrm{GL}_q(\K)$ such that $\tau (h_2) = A h_1$.
\end{proposition}

\begin{proof}
    We established in the proof of \cref{thm:output-realization} that, if $(\rho,h)$ is an output realization of $(\fg, \fp)$, then $\codim_\fg \fhp \leq \codim_\fg \ker (E_\rho) \leq d$, where $E_\rho : \fg \to \K^d$ was defined by $E_\rho(g) := \rho(g)(0)$.
    Since $d = \codim_\fg \fhp$, $\ker E_{\rho_1} = \ker E_{\rho_2} = \fhp$.
    Hence $\rho_1$ and $\rho_2$ are two realizations of the same transitive pair $(\fg, \fhp)$.
    By \cref{thm:uniqueness}, there exists an automorphism $\tau$ of $\kx$ satisfying \eqref{eq:tau-rho}.

    It remains to compare the outputs. 
    For $i \in \{1,2\}$, recall the linear map $G_i : U(\fg) \to \K^q$ defined by $G_i(a) := (\rho_i(a) h_i)(0)$.
    Since $(\rho_i,h_i)$ is an output realization of $(\fg,\fp)$, $\ker G_i = \fp$.
    Since $\dim \K^q = q = \codim_{U(\fg)} \fp$, there exists $A \in \mathrm{GL}_q(\K)$ such that $G_2 = A G_1$.
    
    Set $Q := \tau(h_2) - A h_1 \in \K^q \lxl$.
    The $\fg$-equivariance relation \eqref{eq:tau-rho} extends naturally to a $U(\fg)$-equivariance relation.
    Thus
    \begin{equation}
        \forall a \in U(\fg), \quad
        \tau\circ\rho_2(a)=\rho_1(a)\circ\tau.
    \end{equation}
    Therefore, for any $a \in U(\fg)$, using \cref{lem:auto-local,lem:local-homo},
    \begin{equation}
        \begin{split}
            (\rho_1(a) Q) (0) 
            & = (\rho_1(a) \tau (h_2))(0) - (\rho_1(a) A h_1) (0) \\
            & = (\tau (\rho_2(a) h_2))(0) - A (\rho_1 (a) h_1)(0) \\
            & = (\rho_2(a) h_2)(0) - A (\rho_1(a) h_1)(0) = G_2(a) - A G_1 (a) = 0.
        \end{split}     
    \end{equation}
    By \cref{lem:transitive-realization-separates-series}, applied componentwise, we conclude that $Q = 0$ so $\tau(h_2) = A h_1$.
\end{proof}

\subsection{Convergent output realizations}
\label{sec:convergent-output-realizations}

We now give the convergent counterpart of \cref{thm:output-realization}.

Throughout this subsection, we assume that $\fp \subset U(\fg)$ is a linear subspace of finite codimension. 
We fix any norm on $U(\fg) / \fp$.
The following condition is independent of this choice of norm:
\begin{equation}
    \label{eq:output-growth}
    \forall F\subset\fg \text{ finite},\ 
    \exists C_F>0,\
    \forall n\in\N^*,\ 
    \forall g_1,\dotsc,g_n\in F,
    \qquad
    \|\pi_{\fp}(g_1\dotsm g_n)\|_{U(\fg) / \fp}
    \leq
    C_F^n n!.
\end{equation}

\begin{definition}
    An output realization $(\rho,h)$ of $(\fg,\fp)$ is \emph{convergent} when, for all $g \in \fg$, $\rho(g)$ is a convergent formal vector field, and when each component of $h$ is a convergent formal power series.
\end{definition}

We first show that the single estimate \eqref{eq:output-growth} implies the
growth condition required for the convergence of the underlying realization of the transitive pair $(\fg,\fhp)$.

\begin{lemma}
    \label{lem:finite-output-tests}
    Assume that $\codim_\fg \fhp < \infty$. 
    There exist $a_1,\dotsc,a_N\in U(\fg)$ and $K > 0$ such that
    \begin{equation}
        \label{eq:finite-output-tests}
        \forall g \in \fg, \qquad
        \|\chi(g)\|_{\fg/\fhp}
        \leq
        K\sum_{\nu=1}^N \|\pi_{\fp}(g a_\nu)\|_{U(\fg)/\fp}.
    \end{equation}
\end{lemma}

\begin{proof}
    For every $a\in U(\fg)$, define $\Lambda_a : \fg / \fhp \to U(\fg) / \fp$ by $\Lambda_a(\chi(g)) := \pi_\fp(g a)$, which is well-defined because $\fhp U(\fg)\subset\fp$.
    Moreover,
    \begin{equation}
        \bigcap_{a\in U(\fg)} \ker \Lambda_a=\{0\}.
    \end{equation}
    Indeed, if $\chi(g)$ belongs to this intersection, then $\pi_{\fp}(ga)=0$ for every $a\in U(\fg)$, hence $gU(\fg)\subset\fp$, and therefore $g\in\fhp$ by definition of $\fhp$.
    Since $\fg/\fhp$ is finite-dimensional, there exist finitely many elements $a_1,\dotsc,a_N\in U(\fg)$ such that the map
    \begin{equation}
        \Lambda:\fg/\fhp\to (U(\fg)/\fp)^N,
        \qquad
        \Lambda(\chi(g))
        :=
        \big(\pi_{\fp}(ga_1),\dotsc,\pi_{\fp}(ga_N)\big)
    \end{equation}
    is injective. 
    Thus, for some constant $K>0$, \eqref{eq:finite-output-tests} holds.
\end{proof}

\begin{lemma}
    \label{lem:output-growth-implies-isotropy-growth}
    Assume that \eqref{eq:output-realization-condition} and \eqref{eq:output-growth} hold. 
    Then the transitive pair $(\fg,\fhp)$ satisfies the convergence criterion \eqref{eq:growth} of \cref{thm:main}.
\end{lemma}

\begin{proof}
    Let $a_1, \dotsc, a_N$ be given by \cref{lem:finite-output-tests}.
    Write each $a_\nu$ as a finite linear combination of words on $\fg$. 
    Let $A\subset\fg$ be a finite set containing all letters appearing in these words, and let $L\in\N$ be an upper bound on their lengths.
    
    Fix a finite set $F\subset\fg$.
    For $g_1, \dotsc, g_n \in F$, the element
    \begin{equation}
        B_n:=[g_1,[g_2,[\dotsc,[g_{n-1},g_n]\dotsb]]]
    \end{equation}
    is a sum of at most $2^{n-1}$ words of length $n$ in the letters $g_1,\dotsc,g_n$. 
    Hence, for each $\nu$, the product $B_n a_\nu$ is a finite linear combination, with coefficients independent of $n$, of words of length at most $n+L$ in the finite set $F\cup A$.
    Applying \eqref{eq:output-growth} to the finite set $F\cup A$, and absorbing the finitely many coefficients appearing in the $a_\nu$, we   obtain constants $C,M>0$ such that, for every $\nu$,
    \begin{equation}
        \|\pi_{\fp}(B_n a_\nu)\|_{U(\fg)/\fp}
        \leq
        C 2^n M^{n+L}(n+L)! .
    \end{equation}
    Writing $(n+L)! \leq (n-1)! (L+1)! 2^{n+L}$, there exists $C_F'>0$ such that
    \begin{equation}
        \|\pi_{\fp}(B_n a_\nu)\|_{U(\fg)/\fp}
        \leq
        (C_F')^n (n-1)! .
    \end{equation}
    The estimate \eqref{eq:growth} now follows from \eqref{eq:finite-output-tests}.
\end{proof}

\begin{theorem}
    \label{thm:convergent-output-realization}
    Let $\fg$ be a Lie algebra and let $\fp\subset U(\fg)$ be a linear subspace. 
    Then $(\fg,\fp)$ admits a convergent output realization if and only if \eqref{eq:output-realization-condition} and the growth condition \eqref{eq:output-growth} hold.
    When these conditions hold, the canonical output realization $(\rho,h)$ given by \cref{thm:output-realization} is convergent.
\end{theorem}

\begin{proof}
    We first prove the necessity of \eqref{eq:output-growth} (we already know that \eqref{eq:output-realization-condition} is necessary from \cref{thm:output-realization}). 
    Let $(\rho,h)$ be a convergent output realization of $(\fg,\fp)$ in dimensions $(d,q)$. 
    Recall the linear map $G : U(\fg) \to \K^q$ defined by $G(a) := (\rho(a) h)(0)$.
    Since $\ker G = \fp$ and $U(\fg) / \fp$ is finite-dimensional, there exists $C>0$ such that
    \begin{equation}
        \forall a \in U(\fg), \qquad
        \|\pi_{\fp}(a)\|_{U(\fg) / \fp}
        \leq 
        C \|G(a)\|_{\K^q}.
    \end{equation}
    Let $F\subset\fg$ be finite. 
    Since the vector fields $\rho(g)$ for $g\in F$, and the components of $h$ are convergent, there exists $r>0$ such that all these series belong to the corresponding analytic spaces of radius~$r$.
    Applying the product estimate of \cref{thm:prod.analytic} componentwise,
    \begin{equation}
        \|G(g_1\dotsm g_n)\|_{\K^q}
        = \|(\rho(g_1)\dotsm\rho(g_n)h)(0)\|_{\K^q}
        \leq n! \left(\frac{e}{r}\right)^n \left(\max_{g \in F} \opnorm{\rho(g)}_r \right)^n \opnorm{h}_r
    \end{equation}
    for all $n\in\N^*$ and $g_1,\dotsc,g_n\in F$. 
    This entails \eqref{eq:output-growth}.

    \bigskip

    Conversely, assume \eqref{eq:output-realization-condition} and \eqref{eq:output-growth}. By \cref{lem:output-growth-implies-isotropy-growth}, this transitive pair satisfies the convergence criterion \eqref{eq:growth} of \cref{thm:main}. 
    Hence its canonical realization $\rho$ is convergent.
    It remains to prove that the canonical $h$ given by \eqref{eq:canonical-output-h} is convergent.
    Applying \eqref{eq:output-growth} to the finite set $\{b_1,\dotsc,b_d\}$, there exists $C>0$ such that, for every $m\in\N^d$ with $|m|\geq 1$,
    \begin{equation}
        |h_m| = \frac{1}{m!} \|\pi_{\fp}(\mathbf b^m)\|_{U(\fg)/\fp}
        \leq \frac{1}{m!} C^{|m|}|m|!
        \leq (Cd)^{|m|}.
    \end{equation}
    Thus $h$ is convergent.
\end{proof}

\begin{corollary}
    \label{cor:convergent-output-uniqueness}
    Assume that \eqref{eq:output-realization-condition} and     \eqref{eq:output-growth} hold, and set $d:=\codim_{\fg}\fhp$ and $q:=\codim_{U(\fg)}\fp$. 
    Let $(\rho_1,h_1)$ and $(\rho_2,h_2)$ be two convergent output realizations of $(\fg,\fp)$ in dimensions $(d,q)$. 
    Then the automorphism $\tau$ in \cref{prop:output-uniqueness-with-output-matrix} is convergent.
\end{corollary}

\begin{proof}
    This follows from the case $\fh = \fh' = \fhp$ of \cref{thm:convergent-morphism-transitive-pairs}.
\end{proof}

\section{Applications to control theory}
\label{sec:applications}

We give applications of the abstract results of the previous sections to the realization theory for control-affine systems.
Let $\Omega_0 \subset \R^d$ be an open neighborhood of $0$ in $\R^d$ and $\mathfrak{X}(\Omega_0)$ be the set of real-analytic vector fields on $\Omega_0$.
Consider systems of the form
\begin{equation} 
    \label{eq:affine}
    \dot{x}(t) = f_0(x(t)) + \sum_{i=1}^m u_i(t) f_i(x(t)),
\end{equation}
where $f_0, f_1, \dotsc, f_m \in \mathfrak{X}(\Omega_0)$ are given and $u_1, \dotsc, u_m \in L^1((0,T);\R)$ are controls to be chosen.
When well-defined (for example for small enough times and controls), the unique absolutely continuous solution to \eqref{eq:affine} with initial condition $x_0 \in \Omega_0$ is denoted $\gamma_f(t;u,x_0)$.

Systems of this form are both ubiquitous in engineering, physics or biology, and important for theoretical reasons (see classical textbooks on control theory \cite{BulloLewis2005,Coron2007,Jurdjevic1997,Sontag1998}).

\subsection{Background on local controllability}

We recall a few standard facts on the local controllability of system \eqref{eq:affine}, mainly to motivate the coordinate-invariant questions.
Many different notions are used (see \cite[Section 1.2]{BeauchardMarbach2026} or \cite{BoscainCannarsaFranceschiSigalotti2023}).
To fix ideas, a typical setting is to assume that $f_0(0) = 0$ so that $0$ is an equilibrium of the uncontrolled system and to study the small-state small-time local controllability.

\begin{definition}[Small-state STLC]
    \label{def:SS-STLC}
    We say that \eqref{eq:affine} is \emph{small-state small-time locally controllable} from $0$ when, for every $T > 0$, for every $\delta > 0$, there exists $r > 0$ such that, for every target state $x^* \in B(0,r)$, there exists $u \in L^1((0,T);\R^m)$ such that the solution to \eqref{eq:affine} associated with the initial condition $x(0) = 0$ and the control $u$ satisfies $x(T) = x^*$ and $x([0,T]) \subset B(0,\delta)$.
\end{definition}

Many necessary or sufficient conditions for local controllability involve Lie brackets of vector fields.
When $f, g \in \mathfrak{X}(\Omega_0)$ one defines $[f,g] := (Dg) f - (Df) g \in \mathfrak{X}(\Omega_0)$.
Given a family $\mathcal{F} \subset \mathfrak{X}(\Omega_0)$, denote by $\mathcal{L}(\mathcal{F}) \subset \mathfrak{X}(\Omega_0)$ the Lie algebra spanned by iterated Lie brackets of elements of $\mathcal{F}$.

A fundamental result is the following necessary and sufficient condition in the driftless case $f_0 = 0$, due to Chow \cite{Chow1939} and Rachevsky \cite{Rashevski1938} and usually called the \emph{Lie algebra rank condition}.

\begin{theorem}
    Assume that $f_0 = 0$.
    Then system \eqref{eq:affine} is small-state STLC from $0$ if and only if
    \begin{equation}
        \label{eq:LARC-bis}
        \dim \big\{ h(0) \mid h \in \mathcal{L}(\{f_1, \dotsc, f_m\}) \big\} = d.
    \end{equation}
\end{theorem}

For control-affine systems with drift ($f_0 \neq 0$), one only knows necessary and sufficient conditions for particular subclasses (e.g.\ for linear \cite[Cor.\ 5.5]{Kalman1960} or odd \cite{Brunovsky1976} systems).
The general case is still an active field of research and many necessary or sufficient conditions are known (see \cite{BeauchardMarbach2026} for a survey of necessary conditions, and \cite{AgrachevGamkrelidze1993_Families,Kawski1987_Survey,Krastanov2009} for sufficient ones).

A very crude sufficient condition is the following one (see \cite[Ch.\ 4, Thm.\ 2]{Jurdjevic1997} or \cite[Cor.\ 2.9]{BeauchardLaurentMarbach2026}).

\begin{proposition}
    Assume that $f_0(0) = 0$ and \eqref{eq:LARC-bis}.
    Then system \eqref{eq:affine} is small-state STLC.
\end{proposition}

Condition \eqref{eq:LARC-bis} is sufficient but not necessary, as illustrated by the system $\dot{x}_1 = u_1$ and $\dot{x}_2 = x_1$.
A very crude necessary condition is the Lie algebra rank condition with $f_0$ (see \cite[Thm.\ 3.17]{Coron2007}).

\begin{proposition}
    Assume that $f_0(0) = 0$ and that system \eqref{eq:affine} is small-state STLC.
    Then
    \begin{equation}
        \label{eq:LARC-f0}
        \dim \big\{ h(0) \mid h \in \mathcal{L}(\{f_0, f_1, \dotsc, f_m\}) \big\} = d.
    \end{equation}
\end{proposition}

Condition \eqref{eq:LARC-f0} is necessary but not sufficient, as illustrated by the system $\dot{x}_1 = u$ and $\dot{x}_2 = x_1^2$.
It nevertheless implies a weaker property called \emph{accessibility} (see \cite[Ch.\ 3, Thm.\ 1]{Jurdjevic1997}).

\begin{proposition}
    \label{prop:LARC=>accessible}
    Assume that \eqref{eq:LARC-f0} holds.
    Let $T > 0$ and define 
    \begin{equation}
        \label{eq:AT0}
        A_{\leq T}(0) := \{ \gamma_f(t;u,0) \mid u \in L^1((0,t);\R^m), \enskip \gamma_f([0,t];u,0) \subset \Omega_0, \enskip t \leq T \}.
    \end{equation}
    Then $A_{\leq T}(0)$ contains open balls arbitrarily close to $0$.
\end{proposition}

\subsection{Control-affine systems as convergent realizations}

We now interpret \eqref{eq:affine} in the language of \cref{sec:realizations-Lie}.
By taking Taylor expansions at the origin, every analytic vector field on $\Omega_0$ defines a convergent formal vector field in $\Der \R\lxl$. 
We shall use the same notation for an analytic vector field and for its Taylor expansion at $0$.

Let $X := \{X_0,X_1, \dots, X_m\}$ be a finite set of  \emph{non-commutative indeterminates}, and $\mathcal{L}(X)$ be the \emph{free Lie algebra} generated by $X$ over $\R$.

Given $f_0, f_1, \dotsc, f_m \in \mathfrak{X}(\Omega_0)$, define the homomorphism of Lie algebras 
\begin{equation}
    \rho_f : 
    \begin{cases}
        \mathcal{L}(X) & \to \Der \R\lxl \\
        X_i & \mapsto f_i.
    \end{cases}
\end{equation}
Thus, for $b \in \mathcal{L}(X)$, the vector field $\rho_f(b)$ is obtained by replacing each indeterminate $X_i$ by $f_i$.  For example, $\rho_f([X_1,[X_0,X_1]]) = [f_1,[f_0,f_1]]$.

We also introduce the \emph{evaluation map}
\begin{equation}
    \label{eq:Ef}
    E_f :
    \begin{cases}
        \mathcal{L}(X) &\to \R^d, \\
        b &\mapsto \rho_f(b)(0).
    \end{cases}
\end{equation}

\begin{lemma}
    With these notations, $\rho_f$ is a convergent realization in $d$ variables of the transitive pair $(\mathcal{L}(X),\ker E_f)$.
    It is a realization in minimal dimension if and only if $E_f$ is onto.
\end{lemma}

\begin{proof}
    Since the vector fields $f_i$ are analytic, every iterated Lie bracket $\rho_f(b)$ is analytic, hence defines a convergent formal vector field.
    Moreover,
    \begin{equation}
        \rho_f(b) \in \Der_0 \R \lxl
        \iff 
        \rho_f(b)(0) = 0
        \iff 
        b \in \ker E_f.
    \end{equation}
    Thus $\rho_f$ is a realization of $(\mathcal{L}(X),\ker E_f)$.
    Finally, $\codim_{\mathcal{L}(X)} \ker E_f = \rank E_f$, so the realization is minimal in dimension $d$ exactly when    $\rank E_f=d$.
\end{proof}

\subsection{Equivalence of control systems}

Small-state STLC (as well as many other control properties) is clearly invariant under a local change of coordinates.
This observation motivates the question: \emph{when are two control-affine systems of the form \eqref{eq:affine} equivalent} (in the sense that they can be related by a change of coordinates)?
And more precisely, how can one determine this from the vector fields involved?

The following result is due to Krener \cite[Theorem 1]{Krener1973}.
We re-interpret it using \cref{sec:realizations-Lie,sec:coordinate-changes}.

\begin{proposition}
    \label{prop:Krener}
    Let $d, d' \in \N^*$, $f_0, f_1, \dotsc, f_m$ analytic vector fields defined on a neighborhood of~$0$ in $\R^d$ and $g_0, g_1, \dotsc, g_m$ analytic vector fields defined on a neighborhood of $0$ in $\R^{d'}$.
    Assume that the maps $E_f$ and $E_g$ of \eqref{eq:Ef} are onto.
    The following statements are equivalent.
    \begin{enumerate}
        \item \label{it:krener-1} 
        $\ker E_f \subset \ker E_g$.
        
        \item \label{it:krener-2} 
        There exists a linear map $L : \R^d \to \R^{d'}$ such that, for all $b \in \mathcal{L}(X)$, $\rho_g(b)(0) = L \rho_f(b)(0)$.

        \item \label{it:krener-3} 
        There exists an open neighborhood $\Omega$ of $0$ in $\R^d$ and an analytic map $\lambda : \Omega \to \R^{d'}$ satisfying $\lambda(0) = 0$ and $d \lambda(0) = L$ such that, for all $i \in \intset{0,m}$ and $x \in \Omega$, $d \lambda(x) f_i(x) = g_i(\lambda(x))$.
        
        \item \label{it:krener-4} 
        There exists an open neighborhood $\Omega$ of $0$ in $\R^d$ and an analytic map $\lambda : \Omega \to \R^{d'}$ satisfying $\lambda(0) = 0$ and $d \lambda(0) = L$ such that, for all $T > 0$ and $u \in L^1((0,T);\R^m)$ for which $\gamma_f([0,T];u,0) \subset \Omega$, then $\gamma_g(t;u,0) = \lambda(\gamma_f(t;u,0))$ for all $t \in [0,T]$. 
    \end{enumerate}
    When these hold, we say that the $g$-system is \emph{embedded} in the $f$-system.
\end{proposition}

\begin{proof}
    The equivalence between \cref{it:krener-1,it:krener-2} is a basic exercise in linear algebra.
    By a standard exercise in differential geometry (see \cite[Prop.\ 8.30]{Lee2013}), \cref{it:krener-3} is also equivalent to the fact that $d\lambda(x) \rho_f(b)(x) = \rho_g(b)(\lambda(x))$ for all $b \in \mathcal{L}(X)$.
    The equivalence between \cref{it:krener-1,it:krener-3} is thus precisely \cref{thm:convergent-morphism-transitive-pairs}.
    
    \cref{it:krener-3} implies \cref{it:krener-4} because $\gamma_g(t;u,0)$ and $\lambda(\gamma_f(t;u,0))$ satisfy the same ODE when $d\lambda \cdot f_i = g_i \circ \lambda$.
    Conversely, the equality of trajectories $\gamma_g = \lambda (\gamma_f)$ implies that, for any $i \in \intset{0,m}$, any $T > 0$, and any $x \in A_{\leq T}(0)$ (defined in \eqref{eq:AT0}), one has $d\lambda(x) f_i(x) = g_i(\lambda(x))$.
    By \cref{prop:LARC=>accessible}, $A_{\leq T}(0)$ contains an open ball.
    By analyticity, we thus have $d\lambda(x) f_i(x) = g_i(\lambda(x))$ everywhere.
\end{proof}

\cref{prop:Krener} implies that a control-affine system is entirely characterized by $\ker E_f$.
The role of the ``set of relations'' $\ker E_f$ was already highlighted by Sussmann in \cite[Theorem D]{Sussmann1985}.

\begin{corollary}
    \label{cor:kernel}
    Let $f_0, f_1, \dotsc, f_m$ and $g_0, g_1, \dotsc, g_m$ be analytic vector fields defined on a neighborhood of~$0$ in $\R^d$.
    Assume that $E_f$ and $E_g$ are onto.
    Then the two systems are analytically diffeomorphic if and only if $\ker E_f = \ker E_g$.
\end{corollary}

\subsection{Realization of control-affine systems}

By \cref{cor:kernel}, a control-affine system is entirely characterized by the kernel of the associated evaluation map $E_f$.
This observation motivates the characterization of the linear maps of the form~$E_f$ within the space of linear maps $\mathcal{L}(X) \to \R^d$.

Our main abstract result \cref{thm:main} yields the following equivalence.

\begin{theorem}
    \label{thm:realization-control}
    Let $E : \mathcal{L}(X) \to \R^d$ be a surjective linear map.
    The following are equivalent:
    \begin{enumerate}
        \item There exist a neighborhood $\Omega_0 \subset \R^d$ of $0$ and $f_0, f_1, \dotsc, f_m \in \mathfrak{X}(\Omega_0)$ such that $E = E_f$.

        \item $\ker E$ is a Lie subalgebra of $\mathcal{L}(X)$ and
        \begin{equation}
            \label{eq:realization-control-growth}
            \begin{split}
                \forall F \subset \mathcal{L}(X) \text{ finite}, \enskip
                \exists C_F>0, \enskip
                & \forall n\in\N^*, \enskip
                \forall b_1,\dotsc,b_n\in F, \\
                &
                \| E ([b_1, [\dotsc, [b_{n-1}, b_n]\dotsb]]) \|_{\R^d}
                \leq
                C_F^n (n-1)!.
            \end{split}
        \end{equation}
    \end{enumerate}
\end{theorem}

Let us give a typical real-analytic example.

\begin{example}
    Take $d = 2$ and consider the real-analytic affine system on $\Omega_0 = (-1,1)^2 \subset \R^2$:
    \begin{equation}
        \dot{x}_1 = u(t), \qquad \dot{x}_2 = \frac{x_1}{1-x_1},
    \end{equation}
    which is of the form \eqref{eq:affine} with $f_0(x) = \frac{x_1}{1-x_1} e_2$ and $f_1(x) = e_1$.
    
    Let $X = \{ X_0, X_1 \}$ and $I$ be the ideal of $\mathcal{L}(X)$ spanned by brackets containing at least two $X_0$.
    The map $E_f$ satisfies
    \begin{equation}
        E_f(X_0) = 0, \quad 
        E_f(X_1) = e_1, \quad 
        E_f(\ad_{X_1}^n (X_0)) = n! e_2 \quad (n \geq 1), \quad 
        E_f \vert_{I} = 0.
    \end{equation}
    Thus
    \begin{equation}
        \ker E_f = \R X_0 \oplus \vect \{ \ad_{X_1}^n(X_0) - n! [X_1,X_0] \mid n \geq 2 \} \oplus I.
    \end{equation}
    One easily checks that the right-hand side is indeed a Lie subalgebra of $\mathcal{L}(X)$.

    Let us check that $E_f$ satisfies the growth condition \eqref{eq:realization-control-growth}.
    For $b \in \mathcal{L}(X)$, write
    \begin{equation}
        b \equiv \alpha(b) X_1 + \sum_{k = 0}^{K(b)} \beta_k(b) \ad_{X_1}^k(X_0)
        \quad \pmod{I}.
    \end{equation}
    Then, for any $n \geq 2$,
    \begin{equation}
        [b_1, [\dotsc, [b_{n-1}, b_n]\dotsb]]
        \equiv
        \sum_{j=1}^n \left(\prod_{i \neq j} \alpha(b_i)\right) (-1)^{n-j} \sum_{k = 0}^{K(b_j)} \beta_k(b_j) \ad_{X_1}^{n-1+k}(X_0)
        \quad \pmod{I}.
    \end{equation}
    Let $F \subset \mathcal{L}(X)$ be a finite family.
    Let $A_F := \max_{b \in F}  |\alpha(b)| + \sum_{k=0}^{K(b)} |\beta_k(b)|$ and $K_F := \max_{b \in F} K(b)$.
    Hence
    \begin{equation}
        \| E_f([b_1, [\dotsc, [b_{n-1}, b_n]\dotsb]]) \| 
        \leq n A_F^n (n-1+K_F)!
        \leq A_F^n 2^{n+K_F} (K_F+1)! (n-1)!
    \end{equation}
    which is of the form \eqref{eq:realization-control-growth}.
\end{example}

As a counterexample, let us consider a system which is not real-analytic.

\begin{example}
    \label{ex:non-conv}
    Let $d = 3$.
    By Borel's lemma, there exists $\varphi \in C^\infty(\R;\R)$ such that $\varphi(0) = 0$ and, for every $k \in \N^*$, $\varphi^{(k)}(0)=(2k+1)!$.
    Let $f_0, f_1:\R^3 \rightarrow \R^3$ be defined by 
    \begin{equation}
        f_0(x) := 
        \begin{pmatrix}
            0 \\ x_1 \\ \varphi(x_2)    
        \end{pmatrix}
        \quad \text{and} \quad
        f_1(x) := 
        \begin{pmatrix}
            1 \\ 0 \\ 0    
        \end{pmatrix}.
    \end{equation}
    In particular, $f_0$ is not real-analytic in any neighborhood of $0$.
    
    With the notation $M_1 := [X_1, X_0]$, we have
    \begin{equation}
        E_f(X_0) = 0, \quad
        E_f(X_1) = e_1, \quad
        E_f(M_1) = e_2, \quad
        E_f(\ad_{M_1}^k(X_0)) = \varphi^{(k)}(0) e_3
        \quad (k \geq 1).
    \end{equation}
    Consider the family $F := \{ X_0, M_1 \}$.
    Then, for $k \geq 1$,
    \begin{equation}
        \| E_f(\ad_{M_1}^k(X_0)) \|
        = \| E_f([M_1, [M_1, \dotsc, [M_1, X_0]\dotsb]]) \|
        = (2k+1)!
    \end{equation}
    and thus cannot be bounded by $C_F^{k+1} k!$ for any finite constant $C_F$.
    So \eqref{eq:realization-control-growth} is not satisfied.
\end{example}

\subsection{Existence of embedded systems}

The goal of this section is to recover and clarify results obtained in \cite[Section 10.4]{BeauchardMarbach2026} concerning the existence of embedded systems.
First, the following result is a direct consequence of the realization theory, whereas it required a dedicated proof in \cite[Prop.\ 10.11]{BeauchardMarbach2026}.

\begin{proposition} \label{PropB5}
    Let $f_0, f_1, \dotsc, f_m \in \mathfrak{X}(\Omega_0)$ with $\rank E_f = d$.
    Let $I$ be an ideal of $\mathcal{L}(X)$ and $d' := \codim_{\R^d} E_f(I)$.
    Then there exists an embedded system on $\R^{d'}$ such that $\ker d\lambda(0) = E_f(I)$.
\end{proposition}

\begin{proof}
    Let $\fh := \ker E_f$ and $\fh' := \fh + I$.
    Since $\fh$ is a Lie subalgebra of $\mathcal{L}(X)$ and $I$ is an ideal, $\fh'$ is a Lie subalgebra of $\mathcal{L}(X)$.
    By \cref{thm:main}, since the vector fields $f_i$ are analytic, the pair $(\mathcal{L}(X),\fh)$ satisfies the growth condition \eqref{eq:growth}.
    By \cref{lem:growth-mono}, since $\fh \subset \fh'$, this also holds for the pair $(\mathcal{L}(X),\fh')$.
    Moreover, since $\rank E_f = d$, $\codim_{\mathcal{L}(X)} \fh' = \codim_{\R^d} E_f(I) = d'$.
    By \cref{thm:main}, the pair $(\mathcal{L}(X),\fh')$ has a convergent realization $\rho$ in minimal dimension $d'$.
    For $0 \leq i \leq m$, let $g_i := \rho(X_i)$.
    The $g_i$ are real-analytic vector fields in a neighborhood of $0 \in \R^{d'}$.
    By \cref{prop:Krener}, the $g$-system is embedded in the $f$-system.
    It remains to identify $\ker d\lambda(0)$.
    Let $v \in \R^d$. 
    Since $E_f$ is onto, there exists $b \in \mathcal{L}(X)$ such that $v = E_f(b)$.
    Then $d \lambda(0) v = d\lambda(0) E_f(b) = E_g(b)$.
    Thus $v = E_f(b) \in \ker d\lambda(0)$ if and only if $b \in \ker E_g = \fh'$.
    Hence $\ker d\lambda(0) = E_f(I)$.
\end{proof}

\begin{remark}
    \cref{PropB5} states that $I$ being an ideal of $\mathcal{L}(X)$ is a sufficient condition for the existence of an embedded system such that $\ker d\lambda(0) = E_f(I)$.
    In \cite[Remark.\ 10.13]{BeauchardMarbach2026}, we observed that this condition is not necessary, and that the correct assumption might be linked with realization theory.
    Indeed, the above proof shows that the necessary and sufficient condition is that $\fh + I$ is a Lie subalgebra.
\end{remark}

Such a result allows to extract a simpler subsystem from a given system.
In \cite[Section 10.4]{BeauchardMarbach2026}, we used such arguments to prove necessary conditions for small-time local controllability in a very low regularity setting.

\begin{example}
    Let $d = 3$ and consider the system
    \begin{equation}
        \begin{cases}
            \dot{x}_1 = u \\
            \dot{x}_2 = x_1^2 + u x_3 + x_1^4 e^{x_1} \\
            \dot{x}_3 = x_1^3 e^{x_1}.
        \end{cases}
    \end{equation}
    In the form \eqref{eq:affine}, one checks that
    \begin{equation}
        E_f(X_0) = 0, \quad 
        E_f(X_1) = e_1, \quad 
        E_f(\ad_{X_1}^2(X_0)) = 2 e_2, \quad 
        E_f(\ad_{X_1}^3(X_0)) = 6 e_3
    \end{equation}
    and that no other bracket yields contributions on $e_1$ or $e_2$.
    Consider $I$ the Lie ideal of $\mathcal{L}(X)$ spanned by brackets with at least 3 times $X_1$.
    \cref{PropB5} guarantees the existence of an embedded system such that $\ker d\lambda(0) = E_f(I) = \R e_3$.
    Indeed, setting $\lambda(x) := (x_1, x_2 - x_1 x_3)$ leads to the system $\dot{y}_1 = u$ and $\dot{y}_2 = y_1^2$, which is easier to study (in particular, it is nilpotent).
\end{example}

\subsection{Realization of Chen--Fliess series}
\label{subsec:Realisation_CF}

We now apply the output-realization results of \cref{sec:output-realizations} to the classical realization problem for Chen--Fliess series.
We first present this expansion in \cref{Prop:CF_CV_affine}.

\subsubsection{Convergence of the Chen--Fliess expansion}

Let $\mathcal{A}(X) := U(\mathcal{L}(X))$, which we identify with the free associative algebra generated by $X$.
Its canonical basis is the set $X^*$ of words in the letters $X_i$.
For a word \(\omega=X_{j_1}\dotsm X_{j_\ell}\in X^*\),
\begin{itemize}
    \item for vector fields $f_0,f_1,\dotsc, f_m$, we denote by $f_{\omega}$ the differential operator $f_{j_1} \dotsm f_{j_\ell}$, which acts on scalar-valued functions on $\R^d$, and on vector-valued functions $\Psi=(\psi_1,\dots,\psi_d)$ by $f_{\omega} \Psi := (f_{\omega} \psi_1,\dots, f_{\omega} \psi_{d})$;
    
    \item for $t>0$ and $u=(u_1,\dotsc,u_m)\in L^1((0,t);\R^m)$, we denote by $\int_0^t u_{\omega}$ the real number
    \begin{equation}
        \int_0^t u_\omega
        :=
        \int_{\Delta^\ell(t)}
        u_{j_1}(\tau_1)\dotsm u_{j_\ell}(\tau_\ell)\dd \tau,
    \end{equation}
    where $u_0 \equiv 1$ and $\Delta^{\ell}(t) := \{ (\tau_1, \dotsc, \tau_{\ell}) \in (0,t)^{\ell} ; \enskip 0 < \tau_1 < \dotsb < \tau_{\ell} < t \}$.
\end{itemize}
For the empty word $\emptyset$, we set $f_\emptyset = \Id$ and $\int_0^t u_\emptyset = 1$.

One has the following local convergence result for the Chen--Fliess expansion (see for instance \cite[Proposition 136]{BeauchardLeBorgneMarbach2023}, \cite[Proposition~3.37]{Coron2007}, or \cite[Proposition~4.3]{Sussmann1983}).

\begin{proposition}
    \label{Prop:CF_CV_affine}
    Let $f_0,\dotsc,f_m \in \mathfrak{X}(\Omega_0)$ and $h \in C^\omega(\Omega_0;\R^q)$.
    For $t>0$, $u\in L^1((0,t);\R^m)$ and $x_0 \in \Omega_0$ small enough,
    \begin{equation}
        \label{CF_affine_serie}
        h(\gamma_f(t;u,x_0))
        =
        \sum_{\omega\in X^*}
        \left(\int_0^t u_\omega\right)
        (f_\omega h)(x_0),
    \end{equation}
    and the series converges absolutely.
\end{proposition}

In particular, choosing $h = \Id : \R^d \to \R^d$ yields an explicit expansion of the state $\gamma_f(t;u,x_0)$.

Expansion \eqref{CF_affine_serie} suggests the following idea.
Given a family of coefficients $(g_\omega)_{\omega \in X^*} \subset \R^q$, one can consider the operator
\begin{equation}
    u \mapsto \sum_{\omega \in X^*} \left(\int_0^t u_\omega\right) g_\omega.
\end{equation}
Such operators were introduced by Fliess in \cite{Fliess1981} under the name \emph{causal functionals} and are now widely studied under the name \emph{Fliess operators} (see e.g.\ \cite{GrayEbrahimiFard2021}).

\subsubsection{Output realization}

The goal is to characterize the sequences $(g_{\omega})_{\omega \in X^*} \subset \R^q$, for which there exist analytic vector fields $f_0,f_1,\dots,f_m$ and an analytic map $h$ such that, for every $\omega \in X^*$, $(f_{\omega}h)(0)=g_{\omega}$.
For convenience, we identify any sequence $(g_{\omega})_{\omega \in X^*} \subset \R^q$ with the unique linear map $G:\mathcal{A}(X) \rightarrow \R^q$ such that, for every $\omega \in X^*$, $G(\omega)=g_{\omega}$.

We assume below that $G(1)=0$, which amounts to fixing $h(0)=0$.
Define
\begin{equation}
    \label{eq:def-fpG-fhG}
    \fp_G:=\ker G,
    \qquad
    \fh_G:=
    \{b\in\mathcal{L}(X)\mid G(ba)=0
    \text{ for every }a\in\mathcal{A}(X)\}.
\end{equation}
Thus $\fh_G$ is precisely the output isotropy associated with the output kernel $\fp_G$.
Our results of \cref{sec:output-realizations} entail the following result (see \cref{rk:biblio-CF} for bibliographic comments).

\begin{theorem}
    \label{thm:CF-output-realization}
    Let $G:\mathcal{A}(X)\to\R^q$ be a surjective linear map such that $G(1)=0$.
    The following assertions are equivalent.
    \begin{enumerate}
        \item There exist $d\in\N$, a neighborhood $\Omega_0\subset\R^d$ of $0$, analytic vector fields $f_0,\dotsc,f_m\in\mathfrak{X}(\Omega_0)$, and an analytic map $h:\Omega_0\to\R^q$, with $h(0)=0$, such that
        \begin{equation}
            \label{eq:realize-G-output}
            \forall \omega\in X^*,
            \qquad
            (f_\omega h)(0)=G(\omega).
        \end{equation}

        \item The output isotropy $\fh_G$ has finite codimension in $\mathcal{L}(X)$, and $G$ satisfies the analytic growth condition
        \begin{equation}
            \label{eq:CF-growth}
            \forall F\subset\mathcal{L}(X) \text{ finite},\enskip
            \exists C_F>0,\enskip
            \forall n\in\N^*,\enskip
            \forall b_1,\dotsc,b_n\in F, \quad
            \|G(b_1\dotsm b_n)\|_{\R^q}
            \leq C_F^n n!.
        \end{equation}
    \end{enumerate}
    When these hold, one may take $d=\codim_{\mathcal{L}(X)}\fh_G$.
    More explicitly, choose $b_1,\dotsc,b_d\in\mathcal{L}(X)$ whose classes form a basis of $\mathcal{L}(X)/\fh_G$, let $\rho$ be the Blattner--Draisma realization of $(\mathcal{L}(X),\fh_G)$, set $f_i := \rho(X_i)$ for $0 \leq i \leq m$ and set
    \begin{equation}
        \label{eq:canonical-CF-output}
        h(x)
        :=
        \sum_{m\in\N^d}
        G(\mathbf b^m)
        \frac{x^m}{m!},
        \qquad
        \mathbf b^m:=b_d^{m_d}\dotsm b_1^{m_1}.
    \end{equation}
    Then the vector fields $f_i$ and $h$ are analytic on a neighborhood $\Omega_0$ of $0$, and \eqref{eq:realize-G-output} holds.

    Moreover, this minimal realization is locally unique up to analytic change of coordinates: if $\tilde f_0,\dotsc,\tilde f_m$ and $\tilde h$ form another analytic realization of $G$ in dimension $d$, then there exists a unique local analytic diffeomorphism $\lambda$, with $\lambda(0)=0$, such that $d\lambda \cdot f_i = \tilde{f}_i \circ \lambda$ and $h = \tilde{h} \circ \lambda$.
\end{theorem}

\begin{proof}
    Since $\mathcal{A}(X)=U(\mathcal{L}(X))$, this is exactly
    \cref{thm:convergent-output-realization} applied with $\fg = \mathcal{L}(X)$ and $\fp = \fp_G$.
    The condition $G(1)=0$ gives $1\in\fp_G$, and $\codim_{\mathcal{A}(X)}\fp_G=\rank G \leq q$.
    Moreover, the output isotropy $\fh_{\fp_G}$ is precisely $\fh_G$.
    Finally, the quotient $\mathcal{A}(X)/\fp_G$ identifies with $\im G$, so \eqref{eq:output-growth} is equivalent to \eqref{eq:CF-growth}, up to the choice of an equivalent norm on a finite-dimensional space.
    Formula \eqref{eq:canonical-CF-output} is the canonical output formula \eqref{eq:canonical-output-h} written through the map $G$.

    Analytic uniqueness follows from \cref{prop:output-uniqueness-with-output-matrix,cor:convergent-output-uniqueness}.
\end{proof}

\begin{remark}
    \label{rk:biblio-CF}
    \cref{thm:CF-output-realization} is stated and proved by Reutenauer in \cite{Reutenauer1986} in the case of a scalar output ($q = 1$).
    Reutenauer's work is based on Fliess' paper \cite[Theorem I.2]{Fliess1983} (see also \cite[Theorem 2.1]{Jakubczyk1986}).
    Both \cite{Fliess1983} and \cite{Jakubczyk1986} rely on the simpler growth condition:
    \begin{equation}
        \label{eq:growth-Fliess}
        \exists C > 0,
        \forall \omega \in X^*, \quad 
        \| G(\omega) \| \leq C^{|\omega|+1} |\omega|!
    \end{equation}
    Unfortunately, this weaker assumption is not enough to guarantee the analyticity of the realization (see the counterexample given in \cref{sec:hypo}).

    As noted by Oussous in \cite{JacobOussous1990,Oussous1991}, the results of Fliess and Reutenauer were not entirely explicit.
    Here, both the vector fields of \eqref{eq:formula-blattner} and the output of \eqref{eq:canonical-CF-output} can be computed coefficient-wise.
    Such explicit formulas were present in \cite{GrossmanLarson1992} (without the convergence analysis).
\end{remark}

Let us give a toy example, illustrating the difference between the dimension $q$ of the output and the dimension $d$ of the state.

\begin{example}
    Take $q = 1$ and consider the linear map $G : \mathcal{A}(X) \to \R^q$ given by $G(X_1^2 X_0) = 1$ and $G(\omega) = 0$ for all other $\omega \in X^*$.
    One checks that
    \begin{equation}
        \mathcal{L}(X) = \fh_G \oplus \R b_1 \oplus \R b_2
        \quad \text{with} \quad 
        b_1 := X_1, \quad b_2 := \ad_{X_1}^2(X_0)
    \end{equation}
    so that $d = \codim_{\mathcal{L}(X)} \fh_G = 2$.
    The associated canonical realization is given by $f_0(x) = (0,\frac 12 x_1^2)$, $f_1(x) = (1,0)$ and the output $h(x) = x_2$. 
\end{example}

\subsubsection{State realization}

\cref{thm:CF-output-realization} realizes the coefficients as the output of a finite-dimensional analytic system.
One can ask when this output realization is actually a state realization, meaning that the output is the inclusion of a submanifold of the ambient coefficient space.

\begin{corollary}
    \label{cor:CF-state-realization}
    Let $G:\mathcal{A}(X)\to\R^q$ satisfy the equivalent conditions of \cref{thm:CF-output-realization}.
    Set $E := G|_{\mathcal{L}(X)}$ and $d := \rank E$.
    The following assertions are equivalent.
    \begin{enumerate}
        \item There exist a local analytic submanifold $M \subset \R^q$, with $0 \in M$ and $\dim M = d$, and analytic vector fields $f_0,\dotsc,f_m$ on $M$ such that
        \begin{equation}
            \label{eq:state-realization-G}
            \forall \omega\in X^*,
            \qquad
            f_\omega(\Id)(0)=G(\omega),
        \end{equation}
        where $\Id : M \hookrightarrow \R^q$ denotes the inclusion map.

        \item One has $\fh_G = \ker E$, or equivalently,
        \begin{equation}
            \label{eq:state-realization-ideal-condition}
            \forall b\in\ker E, \quad
            \forall a\in\mathcal{A}(X),
            \qquad
            G(ba)=0.
        \end{equation}
    \end{enumerate}
\end{corollary}

\begin{proof}
    Assume first that such a state realization on $M$ is given.
    If $b \in \ker E$, then $f_b(\Id)(0)=G(b)=0$, hence the vector
    $f_b(0)\in T_0M$ vanishes.
    Therefore, for every $a\in\mathcal{A}(X)$,
    \begin{equation}
        G(ba)
        = f_b(f_a \Id)(0)
        = 0,
    \end{equation}
    so $b \in \fh_G$.
    Thus $\ker E \subset \fh_G$.
    The reverse inclusion follows immediately from the definition of $\fh_G$, by taking $a=1$.
    Hence $\fh_G = \ker E$.

    Conversely, if $\fh_G = \ker E$, let $(\rho,h)$ be given by \cref{thm:CF-output-realization}.
    For $b\in\mathcal{L}(X)$,
    \begin{equation}
        dh(0)\rho(b)(0)
        = (\rho(b)h)(0)
        = G(b)
        = E(b).
    \end{equation}
    Since $\rho$ realizes the transitive pair $(\mathcal{L}(X),\fh_G)$, $\rho(b)(0) = 0$ if and only if $b \in \fh_G = \ker E$.
    Thus $dh(0)$ has rank $d$.
    Hence, after possibly restricting the domain of $h$, $M:=h(\Omega_0)$ is an analytic submanifold of $\R^q$, and
    $h : \Omega_0\to M$ is a local analytic diffeomorphism.
    Pushing forward the vector fields $\rho(X_i)$ by $h$ gives analytic vector fields on $M$, and \eqref{eq:state-realization-G} follows by induction on $\omega$. 
\end{proof}

\subsection{Discussion on the growth condition}
\label{sec:hypo}

Inspecting the proofs of \cref{prop:bound-coeff,thm:convergent-output-realization}, we observe that the growth conditions~\eqref{eq:realization-control-growth} of \cref{thm:realization-control} and \eqref{eq:CF-growth} of \cref{thm:CF-output-realization} could be replaced by their restriction to the family 
\begin{equation}
    F := X \cup \{ b_1, \dotsc, b_d \}
\end{equation}
where $b_1, \dotsc, b_d$ are chosen so that their classes span a supplement of $\ker E$ (or $\fh_G$) in $\mathcal{L}(X)$.

It is however insufficient to only consider the family $F = X$, as done in \cite[Condition (C)]{Fliess1983}.
Indeed, consider the linear map $G : \mathcal{A}(X) \to \R$ (so $q = 1$) defined by $G((X_1 X_0)^k X_0) = (2k+1)!$ for all $k \geq 1$ and $G(\omega) = 0$ for all other $\omega \in X^*$.
This map satisfies \eqref{eq:growth-Fliess} with $C = 1$.
However, it corresponds to the output $h(x) := x_3$ of the system of \cref{ex:non-conv} which involves a vector field $f_0$ which is not real-analytic.
Note that this map does not satisfy our stronger condition \eqref{eq:CF-growth} because $G([X_1,X_0]^k X_0) = (2k+1)!$ so the condition fails for the family $F = \{ X_0, [X_1, X_0] \}$.

\subsection{System of the coordinates of the second kind in a Hall basis}

In this paragraph, we highlight the link between the Blattner--Draisma realization formula \eqref{eq:formula-blattner} and Sussmann's inductive formula for the coordinates of the second kind in a Hall basis, introduced in \cite{Sussmann1986} (see also \cite[Section 2.5]{BeauchardLeBorgneMarbach2023}).

\subsubsection{Hall sets}

We follow Viennot's convention of \cite{Viennot1978} for Hall sets, but other conventions exist (e.g.\ in \cite{Reutenauer2003}).

Let $\Br(X)$ denote the free magma over the set $X$.
Elements of $\Br(X)$ can be represented as binary trees, e.g.\ $(X_0,(X_0,X_1))$, when $X_0, X_1 \in X$.
For $b \in \Br(X)$, $|b| \ge 1$ denotes its length, defined by $|b| = 1$ for $b \in X$, and $|(a,b)| = |a| + |b|$.

\begin{definition}
    \label{def:Hall}
    A \emph{Hall set} is a subset $\B$ of $\Br(X)$ endowed with a total order $<$ such that
    \begin{itemize}
        \item $X \subset \B$;
        \item for all $a, b \in \Br(X)$, $(a, b) \in \B$ if and only if
        \begin{equation}
            a, b \in \B, \qquad
            a < b, \qquad 
            \text{and} \qquad
            (b \in X \text{ or } b = (b', b'') \text{ with } b' \leq a);
        \end{equation}
        \item for all $a, b \in \B$ such that $(a,b) \in \B$ then $a < (a,b)$.
    \end{itemize}
\end{definition}

Viennot proved in \cite{Viennot1978} that any Hall set yields a basis of the free Lie algebra $\mathcal{L}(X)$.

\begin{lemma}
    \label{lem:facto}
    Let $\B$ be a Hall set on $X$ and $b \in \B$.
    There exists a unique $r \in \N$, $m_1, \dotsc, m_r \in \N^*$, $a_1 < \dotsb < a_r \in \B$ and $k \in \{0, \dotsc, m\}$, such that
    \begin{equation}
        \label{eq:b-facto}
        b = \ad_{a_r}^{m_r} \dotsc \ad_{a_1}^{m_1}(X_k).
    \end{equation}
\end{lemma}

\begin{proof}
	This is a direct consequence of \cref{def:Hall}, up to iterating the decomposition.
\end{proof}

\begin{definition}
    Let $\B$ be a Hall set on $X$ and $b \in \B$ as in \eqref{eq:b-facto}.
    We call $a_1, \dotsc, a_r$ the \emph{Hall factors} of $b$ and we define $\Fac(b) := \{ a_1, \dots, a_r \} \subset \B$.
    In particular, for $b \in X$, $\Fac(b) = \emptyset$.
    We also define by induction $\Fac^+(b) := \Fac(b) \cup \Fac^+(a_1) \cup \dotsb \cup \Fac^+(a_r)$ and $\Fac^*(b) := \{ b \} \cup \Fac^+(b)$.
    In particular, when $(a,b) \in \B$,
    \begin{equation}
        \label{eq:fac-rec}
        \Fac^*((a,b)) = \{ (a,b) \} \cup \Fac^*(a) \cup \Fac^+(b).
    \end{equation}
\end{definition}

\subsubsection{Infinite product and coordinates of the second kind}

We recall the explicit formula defining coordinates of the second kind indexed by a Hall set $\B$ on a finite set $X$, due to Sussmann in \cite[equation (19)]{Sussmann1986}.

\begin{definition}
    \label{def:xi}
    Let $\B$ be a Hall set on $X = \{ X_0, \dotsc, X_m \}$. 
    Let $T > 0$ and $u \in L^1((0,T);\R^m)$ (with $u_0 \equiv 1$).
    We define the associated coordinates of the second kind by induction for $b \in \B$ by
    \begin{equation}
        \xi_b(0,u) = 0
        \quad \text{and} \quad
        \dot{\xi}_b(t,u) = \frac{\xi_{a_r}^{m_r}}{m_r!} \dotsb \frac{\xi_{a_1}^{m_1}}{m_1!} (t,u) \cdot u_k(t)
    \end{equation}
    when $b = \ad_{a_r}^{m_r} \dotsb \ad_{a_1}^{m_1}(X_k)$ is given by the unique representation of \cref{lem:facto}.
\end{definition}

Sussmann used in \cite{Sussmann1986} these coordinates to express the solutions to formal differential equations and to linear differential equations as infinite products (see also \cite[Section 2.5]{BeauchardLeBorgneMarbach2023}).
We recently proved in \cite{LeBorgneMarbach2026} that this expansion also converges for nonlinear real-analytic vector fields.

\begin{theorem}
    Let $\Omega_0 \subset \R^d$ be an open neighborhood of $0$ and $f_0, \dotsc, f_m \in \mathfrak{X}(\Omega_0)$.
    Let $\B$ be a Hall set on $X = \{ X_0, \dotsc, X_m \}$.
    There exists $\eta > 0$ such that, for any $T > 0$, $u \in L^1((0,T);\R^m)$ and $x_0 \in \R^d$ with $|x_0| \le \eta$ and $T + \|u\|_{L^1} \le \eta$, and for any $t \in [0,T]$,
    \begin{equation}
        \label{eq:thm.main}
        \gamma_f(t;u,x_0) = \left(\overset{\longrightarrow}{\prod_{b \in \B}} e^{\xi_b(t,u) \rho_f(b)}\right)(x_0),
    \end{equation}
    where $\xi_b(t,u) \in \R$ is a coefficient defined in \cref{def:xi}, $\rho_f(b) \in \mathfrak{X}(\Omega_0)$ is an iterated Lie bracket, and $e^{\xi_b(t,u) \rho_f(b)}$ denotes the time-one flow of the autonomous vector field $\xi_b(t,u) \rho_f(b)$ (which is well-defined under the given smallness assumption).
\end{theorem}

\subsubsection{System of the coordinates of the second kind}

We consider the control-affine system of the coordinates of the second kind of a Hall set. 
We highlight its link with realization theory.

\begin{proposition}
    \label{prop:xi-fbeb}
    Let $\B$ be a Hall set on $X = \{ X_0, \dotsc, X_m \}$.
    Let $B$ be a finite subset of $\B$.
    Assume that $B$ is stable by $\Fac$.
    Then $\Xi := (\xi_b)_{b \in B}$ is the solution to a control-affine system
    \begin{equation}
        \label{eq:EDO-XI}
        \dot{\Xi} = F_0(\Xi) + u_1(t) F_1(\Xi) + \dotsb + u_m(t) F_m(\Xi)
        \quad \text{and} \quad
        \Xi(0) = 0,
    \end{equation}
    where, for $0 \leq j \leq m$, with the notations of \cref{lem:facto},
    \begin{equation}
        \label{eq:XI-FJ}
        F_j(\Xi) = \sum_{b \in B, k_b = j} \frac{\xi_{a_r}^{m_r}}{m_r!} \dotsb \frac{\xi_{a_1}^{m_1}}{m_1!} e_b.
    \end{equation}
    Moreover, this system is the Blattner--Draisma realization associated with the complement $B$, ordered by $b_1 < \dotsb < b_d$, of the transitive pair $(\mathcal{L}(X),\fh)$ where $\fh := \vect (\B \setminus B)$.
    
    In particular, for all $b \in \B$, one has
    \begin{equation}
        \rho_F(b)(0) = 
        \begin{cases}
            e_b & \text{if } b \in B \\
            0 & \text{otherwise.}
        \end{cases}
    \end{equation}
\end{proposition}

\begin{proof}
    The fact that $\Xi := (\xi_b)_{b \in B}$ solves \eqref{eq:EDO-XI} with the vector fields \eqref{eq:XI-FJ} is a reformulation of \cref{def:xi}.
    Since $B$ is factor-stable, the coordinates appearing in \eqref{eq:XI-FJ} are indeed part of $\Xi$.

    Since $\B$ is a basis of $\mathcal{L}(X)$, $\mathcal{L}(X) = \vect B \oplus \fh$, and $\fh$ is a Lie subalgebra by \cref{lem:sup-fac-lie} (since $B$ is factor-stable).
    As $B$ is finite, $\codim_{\mathcal{L}(X)} \fh = |B| < +\infty$, so $(\mathcal{L}(X),\fh)$ is a transitive pair.

    Let $\rho$ be given by \eqref{eq:formula-blattner} for the totally ordered basis $(\B,\prec)$ of \cref{def:block-order}, in which $B = \{ b_1 \prec \dotsb \prec b_d \}$ is an initial segment and $\B \setminus B$ is a basis of $\fh$.
    \cref{thm:blattner} applies and $\rho$ is a realization of $(\mathcal{L}(X),\fh)$.
    Identifying the $i$-th coordinate with $\xi_{b_i}$ and splitting each multi-index according to its support, \eqref{eq:formula-blattner} reads
    \begin{equation}
        \rho(X_j) = \sum_{b \in B} \sum_{r \geq 0} \sum_{a_1 < \dotsb < a_r \in B} \sum_{m \in (\N^*)^r} \chi_b^{\prec}(a_r^{m_r} \dotsb a_1^{m_1} X_j) \frac{\xi_{a_r}^{m_r}}{m_r!} \dotsb \frac{\xi_{a_1}^{m_1}}{m_1!} e_b.
    \end{equation}
    By \cref{cor:chi-compat}, the coefficient equals $\chi_b^{\mathrm{H}}(a_r^{m_r} \dotsb a_1^{m_1} X_j)$, which is now computed in the Hall PBW basis.
    By \cref{cor:chi_b=pbw}, it equals $1$ if $b = \ad_{a_r}^{m_r} \dotsb \ad_{a_1}^{m_1}(X_j)$, and $0$ otherwise.
    By the uniqueness in \cref{lem:facto}, for each $b \in B$ with $k_b = j$ exactly one term survives, and its Hall factors belong to $B$ since $B$ is factor-stable.
    Hence $\rho(X_j) = F_j$ for $0 \leq j \leq m$, so $\rho = \rho_F$.

    Finally, let $b \in \B$. If $b \notin B$, then $b \in \fh$, so $\rho_F(b)(0) = 0$.
    If $b \in B$, taking $m = 0$ in \eqref{eq:formula-blattner} gives $\rho_F(b)(0) = \sum_{c \in B} \chi^{\prec}_c(b) e_c = e_b$, since $b$ is an element of the PBW basis associated with $(\B,\prec)$.
\end{proof}

\begin{corollary}
    In particular, in the context of \cref{prop:xi-fbeb}, if $X_0 \notin B$, then $F_0(0) = 0$.
\end{corollary}

\begin{remark}
    Another option to prove \cref{prop:xi-fbeb} is to use that, at the formal level, Sussmann's infinite product 
    \begin{equation}
        S(t) := \left(\overset{\longleftarrow}{\prod_{b \in \B}} e^{\xi_b(t,u) b}\right) \in \widehat{\mathcal{A}}(X)   
    \end{equation}
    is a formal series solution to $\dot{S} = S (X_0 + u_1 X_1 + \dotsb u_m X_m)$ (see \cite[eq.\ (4)]{Sussmann1986}), and compare it with the generating series $\mathcal{E}(x)$ introduced in \eqref{eq:gen-E} for the proof of the realization theorem.
\end{remark}

\subsubsection{Algebraic lemmas for the proof of \cref{prop:xi-fbeb}}

Below, $(\B,<)$ is a Hall set over $X$ and $B = \{ b_1 < \dotsb < b_d \}$ is a finite factor-stable subset of $\B$.

\paragraph{Tails of Hall sets.}
For $q \in \B$, we consider $\B_{\geq q} := \{ c \in \B \mid c \geq q \}$ and its linear span $\fg_{\ge q}$.

\begin{lemma}
    \label{lem:gq-subalgebra}
    For $q \in \B$, $\fg_{\ge q} = \vect \B_{\ge q}$ is a Lie subalgebra of $\mathcal{L}(X)$. 
\end{lemma}

\begin{proof}
    By \cite[Thm.\ 2.1]{BeauchardLeBorgneMarbach2022}, for all $a, b \in \B$, $[a, b] \in \vect \{ c \in \B \mid c > \min (a,b) \}$.
    The claim follows.
\end{proof}

\begin{lemma}
    \label{lem:q-mark-ab}
    Let $q \in \B$ and $a, b \in \B_{\ge q}$ such that $(a,b) \in \B$.
    Then
    \begin{equation}
        \label{eq:q-mark-ab}
        q \in \Fac^*((a,b)) 
        \quad \Longleftrightarrow \quad
        q \in \Fac^*(a) 
        \text{ or }
        q \in \Fac^*(b).
    \end{equation}
\end{lemma}

\begin{proof}
    Since $q \leq a < b$, $q \neq b$.
    Since $q \leq a < (a,b)$, $q \neq (a,b)$.
    So \eqref{eq:fac-rec} proves \eqref{eq:q-mark-ab}.
\end{proof}

\begin{lemma}
    \label{lem:Iq-ideal}
    Let $q \in \B$.
    The subspace $I_q := \vect \{ b \in \B \mid q \in \Fac^*(b) \}$ is an ideal of $\fg_{\geq q}$.
\end{lemma}

\begin{proof}
    For $a < b$, we recall the notion of relative folding of $b$ with respect to $a$, introduced in \cite[Section~4.1]{BeauchardLeBorgneMarbach2022}, as the element $T_a(b) \in \Br(\B)$ defined by induction as
    \begin{equation}
        T_a(b) := 
        \begin{cases}
            b & \text{if } (a,b) \in \B, \\
            \langle T_a(b_1), T_a(b_2) \rangle & \text{if } (a,b) \notin \B \text{ and } b = (b_1, b_2).
        \end{cases}
    \end{equation}
    Since $a < b$, when $(a,b) \notin \B$, one has indeed $b \notin X$ and $b = (b_1, b_2)$ with $a < b_1 < b_2$.

    It is enough, by bilinearity and antisymmetry, to prove that, for any $a<b$ in $\B_{\geq q}$ such that $q\in\Fac^*(a)\cup\Fac^*(b)$, one has $[a, b] \in I_q$.
    By \cite[Theorem 4.12]{BeauchardLeBorgneMarbach2022}, $[a,b]$ is spanned by elements $c \in \B$ which are trees involving $a$ and each leaf of $T_a(b)$ exactly once.
    Iterating the backwards implication of \eqref{eq:q-mark-ab}, it suffices to prove that either $q \in \Fac^*(a)$ or that there exists a leaf $d$ of $T_a(b)$ such that $q \in \Fac^*(d)$.
    We assumed that $q \in \Fac^*(a) \cup \Fac^*(b)$.
    If $q \in \Fac^*(b)$, iterating the direct implication of \eqref{eq:q-mark-ab}, we obtain that there exists a leaf $d$ of $T_a(b)$ such that $q \in \Fac^*(d)$.
\end{proof}

\begin{lemma}
    \label{lem:kq-in-h}
    For $q \in \B \setminus B$, one has $I_q \subset \vect (\B \setminus B)$.
\end{lemma}

\begin{proof}
    By definition, $I_q$ is spanned by the $b \in \B$ such that $q \in \Fac^*(b)$.
    For such a $b$, one has $b \notin B$.
    Otherwise, since $B$ is factor-stable, $\Fac^*(b) \subset B$, so $q \in B$.
    Hence $I_q \subset \vect(\B \setminus B)$.
\end{proof}

\begin{lemma}
    \label{lem:sup-fac-lie}
    The subspace $\vect (\B \setminus B)$ is a Lie subalgebra.
\end{lemma}

\begin{proof}
    Given $a < b \in \B \setminus B$, since $a \in I_a$ and $b \in \fg_{\ge a}$, \cref{lem:Iq-ideal} gives $[a,b] \in I_a$ and $I_a \subset \vect (\B \setminus B)$ by \cref{lem:kq-in-h}.
\end{proof}

\paragraph{Characterization of the PBW basis of a Hall set.}
We prove \cref{cor:chi_b=pbw}.

\begin{lemma}
    \label{lem:Uq-subalgebra}
    For $q \in \B$, the following subspace is a non-unital subalgebra of $\mathcal{A}(X)$:
    \begin{equation}
        U_q := \vect \{ c_r \dotsb c_1 \mid r \geq 1,\ c_r \geq \dotsb \geq c_1 \geq q,\ c_i \in \B \}.
    \end{equation}
\end{lemma}

\begin{proof}
    This follows from \cref{lem:gq-subalgebra}.
\end{proof}

\begin{lemma}
    \label{lem:aRQ-RQ}
    For $q \in \B$, let $\mathcal{R}_q := \sum_{p \leq q} U_p p$.
    For $q \leq a \in \B$, $a \mathcal{R}_q \subset \mathcal{R}_q$. 
\end{lemma}

\begin{proof}
    Let $p\le q$. 
    Since $q\le a$, one has $p\le a$, hence $a\in U_p$.
    As $U_p$ is a non-unital subalgebra by \cref{lem:Uq-subalgebra}, $a(U_p p)=(aU_p)p\subset U_p p$.
    Summing over $p\le q$ gives $a\mathcal R_q\subset\mathcal R_q$.
\end{proof}

We now prove the main ordered contraction relation.

\begin{lemma}
    \label{lem:ordered-contraction}
    Let $b_1 \leq \dotsb \leq b_N \in \B$ and $X_j \in X$.
    Set $T_0 := X_j$ and $T_s := b_s T_{s-1}$ for $1 \le s \le N$.
    Define $H_0:=X_j$, and recursively
    \begin{equation}
        H_s:=
        \begin{cases}
            (b_s,H_{s-1}),
            &
            \text{if }H_{s-1}\in\B
            \text{ and }(b_s,H_{s-1})\in\B,\\
            0,
            &
            \text{otherwise.}
        \end{cases}
    \end{equation}
    Then, for every $1 \le s \le N$, $T_s - H_s \in \mathcal{R}_{b_s}$.
\end{lemma}

\begin{proof}
    We proceed by induction on $s$.
    Assuming that either $s = 0$ or $1 \le s < N$ and $T_s - H_s \in \mathcal{R}_{b_s}$, write
    \begin{equation}
        \label{eq:TSHS}
        T_{s+1} - H_{s+1}
        = b_{s+1} (T_s - H_s) + (b_{s+1} H_s - H_{s+1}).
    \end{equation}
    If $s = 0$, $T_0 - H_0 = X_j - X_j = 0$ so the first term vanishes.
    If $s \ge 1$, since $b_s \le b_{s+1}$, by \cref{lem:aRQ-RQ} and the induction hypothesis, $b_{s+1} (T_s-H_s) \in \mathcal{R}_{b_s} \subset \mathcal{R}_{b_{s+1}}$.
    It remains to treat the second term.
    If $H_s=0$, this term is zero.
    Otherwise, $H_s\in\B$. 
    \begin{itemize}
        \item If $b_{s+1} < H_s$, then $H_{s+1} = (b_{s+1},H_s) \in \B$ and $b_{s+1} H_s - H_{s+1} = H_s b_{s+1} \in \mathcal{R}_{b_{s+1}}$.

        \item If $b_{s+1} \geq H_s$, then $H_{s+1} = 0$ and $b_{s+1} H_s \in \mathcal{R}_{H_s} \subset \mathcal{R}_{b_{s+1}}$.
    \end{itemize}
    This proves that $T_{s+1} - H_{s+1} \in \mathcal{R}_{b_{s+1}}$, which concludes the proof by induction.    
\end{proof}

\begin{lemma}
    \label{cor:chi_b=pbw}
    Let $\B$ be a Hall set over $X$.
    In the PBW basis of $\mathcal{A}(X)$ associated with $\B$, for any $b_1 \leq \dotsb \leq b_N \in \B$ and $X_j \in X$, the projection of $b_N \dotsb b_1 X_j$ onto $\mathcal{L}(X)$ is $[b_N, [\dotsc, [b_1, X_j]\dotsb]]$ if $(b_N,(\dotsc, (b_1,X_j)\dotsb)) \in \B$, and $0$ otherwise.
\end{lemma}

\begin{proof}
    We apply \cref{lem:ordered-contraction}.
    Any element in $\mathcal{R}_{b_N}$ is a product of at least two elements so has zero projection onto $\mathcal{L}(X)$.
    The conclusion follows.
\end{proof}

\paragraph{Compatibility of the two Poincaré--Birkhoff--Witt decompositions.}
We define an order for which $B$ is an initial segment, and compare the two PBW bases.
Let $\fh := \vect (\B \setminus B)$.

\begin{definition}
    \label{def:block-order}
    We define the \emph{block order} $\prec$ on $\B$ as the (non-Hall) total order which agrees with $<$ on $B$ and on $\B \setminus B$, and for which $a \prec b$ for every $a \in B$ and $b \in \B \setminus B$.
\end{definition}

\begin{lemma}
    \label{lem:JB}
    Let $N \ge 1$ and $c_1 \le \dotsb \le c_N \in \B$ such that $c_i \notin B$ for some $i$.
    Then
    \begin{equation}
        c_N \dotsb c_1 \in \fh \mathcal{A}(X).
    \end{equation}
\end{lemma}

\begin{proof}
    Write $c_N \dotsb c_1 = (c_N \dotsb c_{i+1}) q (c_{i-1} \dotsb c_1)$, where $q = c_i \notin B$ and $c_j \geq q$ for $j > i$.
    By \cref{lem:Iq-ideal}, $I_q$ is a Lie ideal of $\fg_{\ge q}$ and hence the right ideal $I_q \mathcal{A}(X)$ of $\mathcal{A}(X)$ is stable under left-multiplication by $\fg_{\ge q}$.
    Since $q \in I_q$, $c_N \dotsb c_1 \in I_q \mathcal{A}(X)$.
    By \cref{lem:kq-in-h}, $I_q \mathcal{A}(X) \subset \fh \mathcal{A}(X)$.
\end{proof}

For $b \in B$, we denote by $\chi_b^{\mathrm{H}}$ and $\chi_b^{\prec}$ the coordinate forms along the basis element $b$ of the PBW bases associated with $(\B,<)$ and with $(\B,\prec)$.

\begin{lemma}
    \label{cor:chi-compat}
    For every $b \in B$, one has $\chi_b^{\prec} = \chi_b^{\mathrm{H}}$.
\end{lemma}

\begin{proof}
    Let $b \in B$. It suffices to compare these two linear maps on the Hall-PBW basis.

    Let $M_B$ be the span of the products $c_N \dotsb c_1$ with $N \in \N$ and $c_N \geq \dotsb \geq c_1 \in B$ (including the empty product~$1$).
    By \eqref{eq:Ug=I+} applied to $(\B,\prec)$, $\mathcal{A}(X) = M_B \oplus \fh \mathcal{A}(X)$.

    Let $a := c_N \dotsb c_1$ where $c_N \ge \dotsb \ge c_1$ be a Hall-PBW monomial.
    \begin{itemize}
        \item 
        If all $c_i \in B$, since $<$ and $\prec$ agree on $B$, $a$ is the same PBW basis element for both orders, so $\chi_b^{\prec}(a) = \chi_b^{\mathrm{H}}(a)$.

        \item 
        If $c_i \notin B$ for some $i$, then $a$ is a Hall-PBW monomial distinct from $b$, so $\chi_b^{\mathrm{H}}(a) = 0$.
        By \cref{lem:JB}, $a \in \fh \mathcal{A}(X)$ and $\chi_b^\prec(a) = 0$.\qedhere
    \end{itemize}
\end{proof}

\appendix

\section{Inverse function theorem for power series}

In \cref{sec:coordinate-changes}, we used the inverse function theorem for (formal or convergent) power series.

\subsection{Formal version}
\label{sec:ift-formal}

The formal version is a classical result (see e.g.\ \cite[Theorem 9.7]{Sambale2023}).
We give a self-contained proof here, to set up the notations for the proof of our convergent version.
We use the following notion.

\begin{definition}[Order]
    For $f \in \kx$, we denote by $\ord(f)$ the least total degree of a monomial occurring in $f$.
    By convention, $\ord(0) = + \infty$.
    In particular, $f \in \mathfrak{m}_x$ if and only if $\ord(f) \geq 1$.
    
    For a vector series $f = (f_1, \dots, f_d) \in (\kx)^d$, we set $\ord(f) = \min_i \ord(f_i)$.
\end{definition}

\begin{lemma}
    \label{lem:HUV}
    Let $H \in (\kx)^d$ with $\ord(H) \geq 2$ and $U, V \in (\ky)^d$ with $\ord(U) \geq 1$ and $\ord(V) \geq 1$.
    Then $\ord(H(U) - H(V)) \geq \ord(U-V) + 1$.
\end{lemma}

\begin{proof}
    Let $\alpha \in \N^d$ with $|\alpha| \geq 2$. 
    Write the difference $U^\alpha-V^\alpha$ as a sum of terms in which one factor is some component of $U-V$, while the remaining $|\alpha|-1$ factors have order at least $1$. 
    Hence every such term has order at least $\ord(U-V)+|\alpha|-1\geq \ord(U-V)+1$.
    Summing over the monomials of $H$ gives the desired estimate.
\end{proof}

\begin{proposition}
    \label{prop:ift-formal}
    Let $\tau : \ky \to \kx$ be a local homomorphism with $d = d'$.
    Then $\tau$ is an isomorphism if and only if the matrix $A := (\partial_{x_j} (\tau y_i) \vert_{x=0})_{1 \leq i, j \leq d}$ is invertible.
\end{proposition}

\begin{proof}
    \emph{Necessity.}
    Assume that $\tau$ is an isomorphism, and let $\eta : \kx \to \ky$ be its inverse.
    Let $B := (\partial_{y_j} (\eta x_i) \vert_{y=0})_{1 \leq i, j \leq d}$.
    Write $\tau y_i = \sum A_{ij} x_j \pmod{\mathfrak{m}_x^2}$ and $\eta x_j = \sum B_{jk} y_k \pmod{\mathfrak{m}_y^2}$.
    Since $\eta \circ \tau = \Id_{\ky}$ and $\tau \circ \eta = \Id_{\kx}$, we obtain $A B = B A = \Id_d$.
    So $A \in \GL_d(\K)$.

    \medskip

    \emph{Sufficiency: construction of a left inverse.}
    Conversely, assume that $A \in \GL_d(\K)$.
    Let $Y(x) := (Y_1(x), \dotsc, Y_d(x))^t$ where $Y_i(x) := \tau y_i$.
    We write $Y$ as its linear part plus higher-order terms:
    \begin{equation}
        Y(x) = A x + H(x) \qquad \text{where} \qquad \ord(H) \geq 2.
    \end{equation}
    We seek a local homomorphism $\eta : \kx \to \ky$ such that $Y(X(y)) = y$ where $X(y) := (X_1(y), \dotsc, X_d(y))$ where $X_i(y) = \eta x_i$.
    Hence, letting $B := A^{-1}$, we must solve the equation
    \begin{equation}
        \label{eq:XBHX}
        X(y) = B y - B H(X(y)).
    \end{equation}
    We solve this equation by formal iteration, letting $X^{(0)}(y) := B y$ and, for $N \geq 0$,
    \begin{equation}
        X^{(N+1)}(y) := B y - B H(X^{(N)}(y)).
    \end{equation}
    For all $N \geq 0$, $\ord (X^{(N)}) \geq 1$.
    Moreover, the homogeneous parts of $X^{(N)}$ stabilize degree by degree.
    Indeed, \cref{lem:HUV} proves by induction that $\ord(X^{(N+1)}-X^{(N)}) \geq N + 2$.
    So we may define a formal vector series $X(y) := \lim_{N \to \infty} X^{(N)}(y)$, which, passing to the limit, solves \eqref{eq:XBHX}.
    Hence we have $\eta \circ \tau = \Id_{\ky}$.

    \medskip \noindent \emph{Sufficiency: the left inverse is a right inverse.}
    It remains to prove that $\tau\circ\eta=\Id_{\kx}$. 
    Set $W(x) := X(Y(x))$.
    Since $Y(X(y))=y$, substituting $y=Y(x)$ gives $Y(W(x)) = Y(x)$.
    Define $\Delta(x):=W(x)-x$.
    From $Y(W(x))=Y(x)$ and $Y(x)=Ax+H(x)$, we obtain
    \begin{equation}
        \Delta(x) = B (H(x) - H(W(x))).
    \end{equation}
    By \cref{lem:HUV}, $\ord(\Delta(x)) \geq \ord(\Delta(x)) + 1$, so $\ord(\Delta(x)) = + \infty$ and $W(x) = x$.
    This proves that $\tau\circ\eta=\Id_{\kx}$, and thus $\tau$ is an isomorphism with inverse $\eta$.
\end{proof}

\subsection{Convergent version}
\label{sec:ift-convergent}

The convergent version is analogous to the inverse function theorem for analytic functions.
With our precise notations, it is studied in \cite[Part II, Chapter II]{Serre1992}.

\begin{proposition}
    \label{prop:ift-convergent}
    Let $\tau:\ky\to\kx$ be an isomorphism.
    Then $\tau$ is convergent (in the sense of \cref{def:conv-homo}) if and only if its inverse is convergent.
\end{proposition}

\begin{proof}
    We use the notations of \cref{prop:ift-formal}.
    Up to a linear change of variables, it suffices to handle the case when $\tau$ is convergent and $A = \Id_d$.
    One needs to prove that the limit $X$ of the sequence $X^{(0)}(y) = y$ and $X^{(N+1)}(y) = y - H(X^{(N)}(y))$ is convergent, where $\ord(H) \geq 2$ and $H$ is convergent.
    Since each $H_i$ is convergent and has order at least $2$, we may choose $s > 0$ such that $\max_i \opnorm{H_i}_s \leq \frac s 2$.
    Indeed, if $H_i \in (\kx)_{s_0}$, then for $0 < s \leq s_0$,
    \begin{equation}
        \opnorm{H_i}_s
        =
        \sum_{|m|\geq 2} |(H_i)_m|s^{|m|}
        \leq
        \left(\frac{s}{s_0}\right)^2\opnorm{H_i}_{s_0},
    \end{equation}
    and therefore $\opnorm{H_i}_s/s\to0$ as $s\to0$.

    We claim that, for every $N \geq 0$, $\max_i \opnorm{X_i^{(N)}}_{s/2} \leq s$.
    This is true for $N = 0$.
    Assume it holds for some $N$.
    Applying the composition estimate \cref{lem:composition-estimate} to $H_i$ and $X^{(N)}$, we get
    \begin{equation}
        \opnorm{X_i^{(N+1)}}_{s/2}
        \leq
        \opnorm{y_i}_{s/2}
        +
        \opnorm{H_i(X^{(N)})}_{s/2}
        \leq
        \frac{s}{2}+\frac{s}{2}
        =
        s.
    \end{equation}
    Passing to the limit yields that $X_i \in (\ky)_{s/2}$, so $\tau^{-1}$ is convergent.
\end{proof}

\section*{Acknowledgments}

The authors acknowledge support from grants ANR-20-CE40-0009 (Project TRECOS), ANR-24-CE40-5470 (Project CHAT) and ANR-11-LABX-0020 (Labex Lebesgue), as well as from the Fondation Simone et Cino Del Duca -- Institut de France.

\bibliographystyle{plain}
\bibliography{control}

\end{document}